\documentclass[journal]{IEEEtran}
\ifCLASSINFOpdf
\else
\fi
\usepackage[utf8]{inputenc} 
\usepackage[T1]{fontenc}      
\usepackage{amsmath} 
\usepackage{amsfonts}  
\usepackage{algorithm}
\usepackage{algpseudocode}
\usepackage{nicefrac}
\usepackage{amsthm} 
\usepackage{graphicx} 
\usepackage{ulem}
\usepackage[hidelinks]{hyperref}

\newtheorem{rk}{Remark}

\newtheorem{ass}{Assumption}
\newtheorem{theo}{Theorem}
\newtheorem{lem}{Lemma}
\newtheorem{definition}{Definition}
\newtheorem{cor}{Corollary}

\newcommand{\Bcal}{\mathcal{B}}

\newcommand{\Fcal}{\mathcal{F}}

\newcommand{\Ncal}{\mathcal{N}}

\newcommand{\Scal}{\mathcal{S}}

\newcommand{\Zcal}{\mathcal{Z}}

\newcommand{\Rset}{\mathbb{R}}

\newcommand{\Nset}{\mathbb{N}}

\newcommand{\Trans}{\scriptscriptstyle\top}
\newcommand{\argm}{\operatorname{argmin}}

\newcommand{\prox}{\operatorname{prox}}
\newcommand{\ri}{\operatorname{ri}}

\DeclareMathOperator*{\minimise}{minimise}

\DeclareMathOperator*{\argmin}{argmin}
\DeclareMathOperator*{\infim}{inf}

\title{A Parametric Non-Convex Decomposition Algorithm for Real-Time and Distributed NMPC}           
\author{Jean-Hubert Hours,~\IEEEmembership{Student Member,~IEEE} and Colin N. Jones,~\IEEEmembership{Senior Member,~IEEE}}

\begin{document}

\maketitle

\begin{abstract}
A novel decomposition scheme to solve parametric non-convex programs as they arise in\ \textit{Nonlinear Model Predictive Control} (NMPC) is presented.~It consists of a fixed number of alternating proximal gradient steps and a dual update per time step. Hence, the proposed approach is attractive in a real-time distributed context.~Assuming that the\ \textit{Nonlinear Program} (NLP) is semi-algebraic and that its critical points are strongly regular, contraction of the sequence of primal-dual iterates is proven, implying stability of the sub-optimality error, under some mild assumptions.~Moreover, it is shown that the performance of the optimality-tracking scheme can be enhanced via a continuation technique.~The efficacy of the proposed decomposition method is demonstrated by solving a centralised NMPC problem to control a DC motor and a distributed NMPC program for collaborative tracking of unicycles, both within a real-time framework.~Furthermore, an analysis of the sub-optimality error as a function of the sampling period is proposed given a fixed computational power.    
\end{abstract}
\let\thefootnote\relax\footnotetext{Jean-Hubert Hours and Colin N. Jones are with the Laboratoire d'Automatique, Ecole Polytechnique F\'ed\'erale de Lausanne (EPFL), CH-$1015$ Lausanne (e-mail:\href{mailto:jean-hubert.hours@epfl.ch}{\nolinkurl{jean-hubert.hours@epfl.ch}},\href{mailto:colin.jones@epfl.ch}{\nolinkurl{colin.jones@epfl.ch}}).\\
The work of Jean-Hubert Hours is supported by the Swiss National Science Foundation (SNSF) under grant $200021$-$135218$. The work of Colin N. Jones is supported by the European Research Council under the European Union's Seventh Framework Programme (FP/$2007$-$2013$), ERC Grant Agreement $307608$.}
\section{Introduction}
\label{sec:intro}

\looseness-1The applicability of NMPC to fast and complex dynamics is hampered by the fact that an NLP, which is generally non-convex, is to be solved at every sampling time. Solving an NLP to full accuracy is not tractable when the system's sampling frequency is high, which is the case for many mechanical or electrical systems.~This difficulty is enhanced when dealing with distributed systems, which consist of several sub-systems coupled through their dynamics, objectives or constraints.~This class of systems typically lead to large-scale NLPs, which are to be solved online as the system evolves. Solving such programs in a centralised manner may be computationally too demanding and may also hamper autonomy of the agents.~Therefore, much research effort is currently brought to develop decentralised computational methods applicable to MPC.~Although several distributed linear MPC algorithms are now available, there only exists very few strategies\ \cite{neco2009,quoc2013} that can address online distributed NMPC programs, as they generally result in non-convex problems.~Most of these techniques essentially consist in fully convexifying the problem at hand and resort to\ \textit{Sequential Convex Programming} (SCP).~Hence, from a theoretical point of view, they cannot be fully considered as distributed non-convex techniques, since the decomposition step occurs at the convex level.~In addition, the SCP methods of\ \cite{neco2009,quoc2013} have not been analysed in an online optimisation framework, where fixing the number of iterations is critical.~In general, it should be expected that SCP techniques require a high number communications between agents, since a sequence of distributed convex NLPs is to be solved.~Therefore, our main objective is to propose a novel online distributed optimisation strategy for solving NMPC problems in a real-time framework.~Our approach is useful for deployment of NMPC controllers, but can also be applied in the very broad parametric optimisation context.~We address the problem of real-time distributed NMPC from a distributed optimisation perspective.~However, some approaches based on distributed NMPC algorithms, which fully decompose the NMPC problem into smaller NMPC problems exist in the literature\ \cite{dun2007,stew2011}.~Yet, to the author's knowledge, they have not been analysed in a real-time setting.~Besides, the approach of\ \cite{stew2011} is limited to the input constrained case. 

\looseness-1A standard way of solving distributed NLPs is to apply Lagrangian decomposition\ \cite{boyd2010}.~This technique requires strong duality, which is guaranteed by Slater's condition in the convex case, but rarely holds in a non-convex setting.~The augmented Lagrangian framework\ \cite{bert1982} mitigates this issue by closing the duality gap for properly chosen penalty parameters\ \cite{rock1974}.~The combination of the bilinear Lagrangian term with the quadratic penalty also turns out to be computationally more efficient than standard penalty approaches, as the risk of running into ill-conditioning is reduced by the fast convergence of the dual sequence.~In addition, global convergence of the dual iterates can be obtained, as in the\ \textsc{lancelot} package\ \cite{conn1988,conn1991}.~However, the quadratic penalty term induces non-separability in the objective, which hampers decomposing the NLP completely.~Several approaches have been proposed to remedy this issue\ \cite{ham2011}.~Taking inspiration from the\ \textit{Alternating Direction Method of Multipliers} (ADMM)\ \cite{boyd2010}, which has recently gained in popularity, even in a non-convex setting (yet without convergence guarantees)\ \cite{jov2013}, we propose addressing the non-separability issue via a novel\ \textit{Block-Coordinate Descent} (BCD) type technique.~BCD or alternating minimisation strategies are known to lead to `easily' solvable sub-problems, which can be parallelised under some assumptions on the coupling, and are well-suited to distributed computing platforms\ \cite{bert1997}.~BCD-type techniques are currently raising interest for solving very large-scale programs\ \cite{nes2012}.~Until very recently, for non-convex objectives with certain structure, convergence of the BCD iterates was proven in terms of limit points only\ \cite{tse2001}.~The central idea of our algorithm is to apply a truncated proximal alternating linearised minimisation in order to solve the primal augmented Lagrangian problem approximately.~Our convergence analysis is based on the recent results of\ \cite{att2009,att2013}, which provide a very general framework for proving global convergence of descent methods on non-smooth semi-algebraic objectives. 

\looseness-1Augmented Lagrangian techniques have proven effective at solving large-scale NLPs \cite{zav2014}.~Yet, in an online context, they are hampered by the fact that a sequence of non-convex programs need to be solved to an increasing level of accuracy\ \cite{bert1982}.~Recently, a parametric augmented Lagrangian algorithm has been introduced by\ \cite{zav2010} for centralised NMPC problems.~ Our online decomposition method builds upon the ideas of\ \cite{zav2010}, but extends them to the online distributed framework, which requires a significantly different approach.~Furthermore, our analysis brings the results of \cite{zav2010} one step further by proving contraction of the sequence of primal-dual iterates under some mild assumptions.~In particular, new insights are given on how the penalty parameter and the number of primal iterations need to be tuned in order to ensure boundedness of the tracking error, which is the core of the analysis of a parametric optimisation algorithm.~Moreover, an interesting aspect of our analysis is that it shows how the proposed decomposition algorithm can be efficiently modified via a continuation, or homotopy, strategy\ \cite{geo1990}, leading to faster convergence of the tracking scheme.~This theoretical observation is confirmed by a numerical example presented later in the paper.~Our technique is also designed to handle a more general class of parametric NLPs, where the primal\ \textit{Projected Successive Over-Relaxation} (PSOR) of \cite{zav2010} is limited to quadratic objectives subject to non-negativity constraints.    

\looseness-1From an MPC perspective, the effect of the sampling period on the behaviour of the combined system-optimiser dynamics has been analysed in a few studies\ \cite{alam2013,alam2014}.~However, it is often claimed that faster sampling rates lead to better closed-loop performance.~This thinking neglects the fact that MPC is an optimisation-based control strategy, which requires a significant amount of computations within a sampling interval.~Hence, sampling faster entails reducing the number of iterations in the optimiser quite significantly, as the computational power of any computing platform is limited, especially in the case of embedded or distributed applications.~In the last part of the paper, we propose an analysis of the effect of the sampling period on the behaviour of the system in closed-loop with our parametric decomposition strategy.~It is demonstrated that the proposed theory accounts for the observed numerical behaviour quite nicely.~In particular, tuning the penalty parameter of the augmented Lagrangian turns out to significantly improve the tracking performance at a given sampling period.    
 
\looseness-1In Section\ \ref{sec:algo_desc}, the parametric distributed optimisation scheme is presented.~In Section\ \ref{sec:theo_tools}, some key theoretical ingredients such as strong regularity of generalised equations and the Kurdyka-Lojasiewicz inequality are introduced, and convergence of the primal sequence is proven.~As explained later in the paper, strong regularity and the Kurdyka-Lojasiewicz inequality encompass a broad class of problems encountered in practice.~Then, in Section\ \ref{sec:con_prop}, contraction of the primal-dual sequence is proven and conditions ensuring stability of the tracking error are derived.~In Section\ \ref{sec:comp_asp}, basic computational aspects are investigated.~Finally, the proposed approach is tested on two numerical examples, which consist in controlling the speed of a DC motor to track a piecewise constant reference for the first one, and collaborative tracking of unicycles for the second one.~The effect of the sampling period on the tracking error, passed through the system dynamics, is analysed given a fixed computational power, or communication rate for the distributed case.
\section{Background}
\label{sec:back}
\begin{definition}[Proximal operator]
\label{def:prox}
Let $f:\Rset^n\rightarrow\Rset\cup\left\{+\infty\right\}$ be a proper lower-semicontinuous function and $\alpha>0$.~The proximal operator of $f$ with coefficient $\alpha$, denoted by $\prox^f_\alpha\left(\cdot\right)$, is defined as follows:
\begin{align}
\label{eq:prox}
\prox^f_\alpha\left(x\right):=\argm_{y} f(y)+\displaystyle\frac{\alpha}{2}\left\|y-x\right\|_2^2\enspace.
\end{align} 
\end{definition}
\begin{definition}[Critical point]
\label{def:crit}
Let $f$ be a proper lower semicontinuous function.~A necessary condition for $x^\ast$ to be a minimiser of $f$ is that
\begin{align}
\label{eq:crit}
0\in\partial f\left(x^\ast\right)\enspace,
\end{align}
where $\partial f\left(x^\ast\right)$ is the sub-differential of $f$ at $x^\ast$\ \cite{rock2009}. 
\end{definition} 
Points satisfying \eqref{eq:crit} are called\ \textit{critical points}. 
\begin{definition}[Normal cone to a convex set]
Let $\Omega$ be a convex set in $\Rset^n$ and $\bar{x}\in\Omega$.~The \textit{normal cone} to $\Omega$ at $\bar{x}$ is the set
\begin{align}
\label{eq:norm_cone}
\Ncal_{\Omega}\left(\bar{x}\right):=\left\{v\in\Rset^n~\big|~\forall x\in\Omega,v^{\Trans}\left(x-\bar{x}\right)\leq0\right\}\enspace.
\end{align} 
\end{definition}
The indicator function of a closed subset $\Omega$ of $\Rset^n$ is denoted by $\iota_\Omega$ and is defined as 
\begin{align}
\iota_{\Omega}\left(x\right)=\begin{cases}0&\mbox{if } x\in\Omega\\
+\infty&\mbox{if }x\notin\Omega\end{cases}\enspace. 
\end{align}
\begin{lem}[Sub-differential of indicator function \cite{rock2009}]
Given a convex set $\Omega$, for all $x\in\Omega$,
\begin{align}
\partial\iota_\Omega\left(x\right)=\Ncal_\Omega\left(x\right)\enspace.
\end{align}
\end{lem}
\begin{lem}[Descent lemma,\cite{bert1982}]
\label{lem:desc_lem}
Let $L:\Rset^n\rightarrow\Rset$ a continuously differentiable function such that its gradient ${\nabla}L$ is $\lambda_L$-Lipschitz continuous.~For all $x,y\in\Rset^n$,
\begin{align}
L(y)\leq L(x)+{\nabla}L(x)^{\Trans}\left(y-x\right)+\displaystyle\frac{\lambda_L}{2}\left\|y-x\right\|_2^2\enspace.
\end{align}
\end{lem}
The distance of a point $x\in\Rset^n$ to a subset $\Sigma$ of $\Rset^n$ is defined by 
\begin{align}
d(x,\Sigma):=\infim_{y\in\Sigma}\big\|x-y\big\|_2\enspace.
\end{align}
The open ball with center $x$ and radius $r$ is denoted by $\Bcal\left(x,r\right)$.~Given $\underline{x},\overline{x}\in\Rset^n$, the box set $\left\{x\in\Rset^n~\big|~\underline{x}\leq x \leq\overline{x}\right\}$ is denoted by $B\left(\underline{x},\overline{x}\right)$.~Given a closed convex set $\Omega\subseteq\Rset^n$, the single-valued projection onto $\Omega$ is denoted by $\pi_\Omega\left(\cdot\right)$.~The relative interior of $\Omega\subseteq\Rset^n$ is denoted by $\ri\Omega$.~Given a polynomial function $f$, its degree is denoted by $\deg\left(f\right)$.~A semi-algebraic function is a function whose graph can be expressed as a union of intersections of level sets of polynomials.  
\section{Solving time-dependent distributed nonlinear programs}
\label{sec:algo_desc}
\subsection{Problem formulation}
\label{subsec:pb}
The following class of parametric NLPs with separable cost, partially separable equality constraints and separable inequality constraints is considered
\begin{align}
\label{eq:prm_nlp_1}
&\minimise_{z_1,\ldots,z_{N_a}}~J\left(z\right):=\sum_{i=1}^{N_a}J_i\left(z_i\right)\\
&\text{s.t.}~~~~Q_c\left(z_1,\ldots,z_{N_a}\right)=0,\enspace\nonumber\\
&~~~~~~~~g_i\left(z_i\right)+T_is_k=0,\enspace\nonumber\\
&~~~~~~~~z_i\in\Zcal_i,~i\in\left\{1,\ldots,N_a\right\},\enspace\nonumber
\end{align}
\looseness-1where $z:=\left(z_1^{\Trans},\ldots,z_{N_a}^{\Trans}\right)^{\Trans}\in\Rset^{n_z}$, with $n_z=\sum_{i=1}^{N_a}n_i\geq2$ and $z_i\in\Rset^{n_i}$.~The vectors $z_i$ model different agents, while the function $Q_c:\Rset^{n_z}\rightarrow\Rset^m$ represent \textit{constraint couplings}.
\begin{rk}
\looseness-1For clarity, the definition of NLP\ \eqref{eq:prm_nlp_1} is restricted to constraint couplings.~However, cost couplings can be addressed by the approach described in the sequel.
\end{rk}
\looseness-1The functions $J_i:\Rset^{n_i}\rightarrow\Rset$ and $g_i:\Rset^{n_z}\rightarrow\Rset^{q_i}$ are \textit{individual cost and constraint functionals} at agent $i\in\left\{1,\ldots,N_a\right\}$.~In an NMPC context, the nonlinear equality constraint involving $g_i$ models the dynamics of agent $i$ over a prediction horizon.~The vector $s_k$ is a time-dependent parameter, which lies within a set $\Scal\subseteq\Rset^p$, where $k$ stands for a time index.~When it comes to NMPC, the parameter $s_k$ stands for a state estimate or a reference trajectory, for instance.~The matrices $T_i\in\Rset^{q_i\times{}p}$ are constant.~The linear dependence of the local equality constraints in the parameter $s$ is not restrictive, as extra variables can be introduced in order to obtain this formulation.~For all $s\in\Scal$, we define the equality constraints functional 
\begin{align}
\nonumber
G(\cdot,s):=\left(Q_c\left(\cdot\right)^{\Trans},\left(g_1\left(\cdot\right)+T_1s\right)^{\Trans},\ldots,\left(g_{N_a}\left(\cdot\right)+T_{N_a}s\right)^{\Trans}\right)^{\Trans}
\end{align}
For all $i\in\left\{1,\ldots,N_a\right\}$, the constraint sets $\Zcal_i$ are assumed to be bounded boxes.~Note that such an assumption is not restrictive, as slack variables can always be introduced.~KKT points of NLP\ \eqref{eq:prm_nlp_1} are denoted by $w_k^{\ast}$ or $w^{\ast}\left(s_k\right)$ without distinction.  
\begin{ass}
\label{ass:ass_1}
The functions $Q_c$, $J_i$ and $g_i$ are multivariate polynomials.~For each $i\in\left\{1,\ldots,N_a\right\}$, $\deg\left(J_i\right)\geq2$.
\end{ass}
\begin{rk}
From a control perspective, this implies that the theoretical developments that follow are valid when NLP\ \eqref{eq:prm_nlp_1} is obtained via discretisation of optimal control problems with polynomial dynamics and quadratic costs for instance. 
\end{rk}
As we are targeting distributed applications, further assumptions are required on the coupling function $Q_c$.
\begin{ass}[Sparse coupling,\cite{bert1997}]
\label{ass:ass_2}
The sub-variables $z_1,\ldots,z_{N_a}$ can be re-ordered and grouped together in such a way that a Gauss-Seidel sweep on the real-valued function $\left\|Q_c\right\|_2^2$ can be performed in $P$ parallel steps, where $P\ll N_a$.~The re-ordered and grouped sub-variables are denoted by $z_1,\ldots,z_P$, so that the re-arranged vector $z$ is defined by $z=\left(z_1^{\Trans},\ldots,z_P^{\Trans}\right)^{\Trans}$.~In the remainder, it is assumed that NLP\ \eqref{eq:prm_nlp_1} has been re-arranged accordingly.  
\end{ass}
\begin{rk}
Assumption\ \ref{ass:ass_2} is standard in distributed computations\ \cite{bert1997}.~It encompasses a large number of practical problems of interest.~For consensus problems, in which coupling constraints $z_1-z_i=0$ appear for all $i\in\left\{2,\ldots,N_a\right\}$, one obtains $P=2$ parallel steps, corresponding to the update of $z_1$ followed by the updates of $z_2,\ldots,z_{N_a}$ in parallel.~When the coupling graph is a tree, such as in the case of a distribution network, one also obtain $P=2$.~Our approach is likely to be more efficient when $P$ is small relative to $N_a$.  
\end{rk}

\subsection{A non-convex decomposition scheme for optimality tracking}
\label{subsec:algo}
At every time instant $k$, a critical point of the parametric NLP\ \eqref{eq:prm_nlp_1} is computed inexactly in a distributed manner.~The key idea is to track time-dependent optima $w^{\ast}_k$ of \eqref{eq:prm_nlp_1} by approximately computing saddle points of the augmented Lagrangian  
\begin{align}
\label{eq:al_1}
L_\rho\left(z,\mu,s_k\right):=~&J\left(z\right)+\left(\mu+\displaystyle\frac{\rho}{2}G\left(z,s_k\right)\right)^{\Trans}G\left(z,s_k\right)\enspace,
\end{align}
subject to $z\in\Zcal$, where $\Zcal:=\Zcal_1\times\Zcal_2\times\ldots\times\Zcal_P$ and $\mu:=\left(\mu_C^{\Trans},\mu_1^{\Trans},\ldots,\mu_P^{\Trans}\right)^{\Trans}\in\Rset^{m+q}$, with $q:=\sum_{i=1}^Pq_i$, is a dual variable associated with the equality constraints $Q_c\left(z\right)=0$, $g_1\left(z_1\right)+T_1s_k=0$,\ldots, $g_P\left(z_P\right)+T_Ps_k=0$.~The scalar $\rho>0$ is the so-called \textit{penalty parameter}, which remains constant as $s_k$ and $\mu$ vary.~As $Q_c$, $J_i$ and $g_i$ are multivariate polynomials, $L_\rho\left(\cdot,\mu,s\right)$ is a multivariate polynomial, whose degree is assumed to be larger than $2$.~We define
\begin{align} 
\label{eq:deg_al}
d_L:=\deg\left(L_\rho\left(\cdot,\mu,s\right)\right)\geq{2}\enspace.
\end{align}
In the rest of the paper, sub-optimality of a variable is highlighted by a $\bar{\cdot}$, and criticality by a ${\cdot}^{\ast}$. We intend to build an approximate KKT point $\left(\bar{z}(s_{k+1})^{\Trans},\bar{\mu}(s_{k+1})^{\Trans}\right)^{\Trans}$ of\ \eqref{eq:prm_nlp_1} by incremental improvement from the KKT point $\left(\bar{z}(s_k)^{\Trans},\bar{\mu}(s_k)^{\Trans}\right)^{\Trans}$, once the parameter $s_{k+1}$ is available. 

\begin{algorithm}[h!]
	\caption{\label{algo:splt_trck}Optimality tracking splitting algorithm}
	\begin{algorithmic}
		\State \textbf{Input:} Suboptimal primal-dual solution $\left(\bar{z}(s_k)^{\Trans},\bar{\mu}(s_k)^{\Trans}\right)^{\Trans}$, parameter $s_{k+1}$, augmented Lagrangian ${L_{\rho}\left(\cdot,\bar{\mu}_k,s_{k+1}\right)}$.
		\State \underline{\textbf{Primal/Inner loop}}:
		\State $z^{(0)}\leftarrow\bar{z}(s_k)$
		\For{$l=0\ldots M-1$}
		\For{$i=1\ldots P$} \Comment{$P\ll N_a$}
		\State \Comment{In parallel in group $i$}
		\State$z_i^{(l+1)}\leftarrow\texttt{bckMin}\left(z_i^{(l)},\bar{\mu}(s_k),\rho,s_{k+1}\right)$ 
		\EndFor
		\EndFor
		\State $\bar{z}(s_{k+1})\leftarrow z^{(M)}$
		\State \underline{\textbf{Dual update}}:~$\bar{\mu}(s_{k+1})\leftarrow\bar{\mu}(s_k)+{\rho}G\left(\bar{z}(s_{k+1}),s_{k+1}\right)$  	
		\end{algorithmic}
\end{algorithm}
Algorithm\ \ref{algo:splt_trck} computes a suboptimal primal variable $\bar{z}(s_{k+1})$ by applying $M$ iterations of a proximal alternating linearised method to minimise the augmented Lagrangian functional $L_\rho\left(\cdot,\bar{\mu}(s_k),s_{k+1}\right)+\sum_{i=1}^P\iota_{\Zcal_i}\left(\cdot\right)$.
    
\looseness-1We define the block-wise augmented Lagrangian function at group $i\in\left\{1,\ldots,P\right\}$, where $P\ll N_a$ by Assumption\ \ref{ass:ass_2},
\begin{align}
\label{eq:def_coor_al}
L_{\rho,\mu,s}^{(i)}:=L_\rho\left(z_1,\ldots,z_{i-1},\cdot,z_{i+1},\ldots,z_P,\mu,s\right)
\end{align}
and a quadratic model at $z_i$, encompassing all agents of group $i$, given a curvature coefficient $c_i>0$,
\begin{align}
q\left(\cdot;z_i,c_i\right):=&L_{\rho,\mu,s}^{(i)}\left(z_i\right)+{\nabla}L_{\rho,\mu,s}^{(i)}\left(z_i\right)^{\Trans}\left(\cdot-z_i\right)\nonumber\\
&~~~~~~~~~~~~~~~~~~~~~~~~~~~~~~~+\displaystyle\frac{c_i}{2}\left\|\cdot-z_i\right\|_2^2\enspace.\nonumber
\end{align}
Given an iteration index $l\geq1$, we define 
\begin{align}
L_{\rho,\mu,s}^{(i,l)}:=L_\rho\left(z_1^{(l+1)},\ldots,z_{i-1}^{(l+1)},\cdot,z_{i+1}^{(l)},\ldots,z_P^{(l)},\mu,s\right)\enspace.\nonumber
\end{align}
For every group of agent indexed by $i\in\left\{1,\ldots,P\right\}$, a regularisation coefficient $\alpha_i>0$ is chosen.~In practice, such a coefficient should be taken as small as possible. 

\begin{algorithm}[h!]
\caption{\label{algo:loc_bck}Parallel backtracking procedure at group $i\in\left\{1,\ldots,P\right\}$ and iteration $l\geq1$, $\texttt{bckMin}\left(z_i,\mu,\rho,s\right)$}
\begin{algorithmic}
	\State \textbf{Input:} Primal variable $z_i^{(l)}\in\Zcal_i$, initial guess on local curvature $c_i>0$, block-wise augmented Lagrangian $L^{(i,l)}_{\rho,\mu,s}$, quadratic model $q$, regularisation coefficient $\alpha_i>0$ and backtracking coefficient $\beta>1$. 
	\State$z_i^{(u)}\leftarrow z_i^{(l)}$
	\State \underline{\textbf{Backtracking loop}}:
	\While{$L_{\rho,\mu,s}^{(i,l)}\left(z_i^{(u)}\right)+\displaystyle\frac{\alpha_i}{2}\left\|z_i^{(u)}-z_i^{(l)}\right\|_2^2>q\left(z_i^{(u)};z_i^{(l)},c_i\right)$} 
		\State$c_i\leftarrow\beta\cdot c_i$
		\State$z_i^{(u)}\leftarrow\prox^{\iota_{\Zcal_i}}_{c_i}\left(z_i^{(l)}-\displaystyle\frac{1}{c_i}\nabla L_{\rho,\mu,s}^{(i,l)}\left(z_i^{(l)}\right)\right)$
	\EndWhile
	\State \textbf{Output:} $z_i^{(u)}$
\end{algorithmic}
\end{algorithm}

\looseness-1Each step of the alternating minimisation among the $P$ groups of agents (Algorithm\ \ref{algo:splt_trck}) consists in backtracking projected gradient steps in parallel for each group (Algorithm\ \ref{algo:loc_bck}).~Later in the paper, it is proven that the backtracking loop of Algorithm\ \ref{algo:loc_bck} terminates in a finite number of iterations, and that convergence of the primal loop in Algorithm\ \ref{algo:splt_trck} to a critical point of the augmented Lagrangian\ \eqref{eq:al_1} is guaranteed for an infinite number of primal iterations ($M=\infty$ in Algorithm \ref{algo:splt_trck}).~In practice, after a fixed number of primal iterations $M$, the dual variable is updated in a first-order fashion.~Hence, the whole procedure yields a suboptimal KKT point 
\begin{align}
\nonumber
\bar{w}_{k+1}=\left(\bar{z}\left(s_{k+1}\right)^{\Trans},\bar{\mu}\left(s_{k+1}\right)^{\Trans}\right)^{\Trans}
\end{align}
for program\ \eqref{eq:prm_nlp_1} given parameter $s_{k+1}$.
\begin{rk}
\looseness-1Incremental approaches are broadly applied in NMPC, as fully solving an NLP takes a significant amount of time and may result in unacceptable time delays.~Yet, existing incremental NMPC strategies\ \cite{diehl2005,zav2009} are based on Newton predictor-corrector steps, which require factorisation of a KKT system.~This a computationally demanding task for large-scale systems that cannot be readily carried out in a distributed context.~Therefore, Algorithm\ \ref{algo:splt_trck} can be interpreted as a distributed incremental improvement technique for NMPC. 
\end{rk}  
\begin{rk}
Note that the active-set at $z^\ast\left(s_{k+1}\right)$ may be different from the active-set at $z^\ast\left(s_k\right)$.~Hence, Algorithm\ \ref{algo:splt_trck} should be able to detect active-set changes quickly.~This is the role of the proximal steps, where projections onto the sets $\Zcal_i$ are carried out.~It is well-known that gradient projection methods allow for fast activity detection\ \cite{conn1988}.
\end{rk}
\section{Theoretical tools: Strong regularity and Kurdyka-Lojasiewicz inequality}
\label{sec:theo_tools}
The analysis of Algorithm\ \ref{algo:splt_trck} is based on the concept of generalised equations, which has been introduced in real-time optimisation by\ \cite{zav2010}.~Another key ingredient for the convergence of the  proximal alternating minimisations in Algorithm\ \ref{algo:splt_trck} is the Kurdyka-Lojasiewicz (KL) property, which has been introduced in nonlinear programming by\ \cite{att2009,att2013} and is satisfied by semi-algebraic functions\ \cite{bolte2007}.~Hence, this property encompasses a broad class of functions appearing in NLPs arising from the discretisation of optimal control problems.

\subsection{Parametric generalised equations}
\label{subsec:pge}
KKT points $w^\ast\left(s\right)=\left({z^\ast(s)}^{\Trans},{\mu^\ast(s)}^{\Trans}\right)^{\Trans}$ of the parametric nonlinear program\ \eqref{eq:prm_nlp_1} satisfy $z^\ast(s)\in\Zcal$ and
\begin{align}
\label{eq:crit_pt}
\left\{
\begin{aligned}
	&0\in\nabla_zJ(z^\ast(s))+\nabla_zG(z^\ast(s),s)^{\Trans}\mu^\ast(s)+\Ncal_{\Zcal}\left(z^\ast(s)\right)\\
	&G(z^\ast(s),s)=0
\end{aligned}
\right.\enspace.
\end{align}
Relation\ \eqref{eq:crit_pt} can be re-written as the generalised equation
\begin{align}
\label{eq:ge_1}
0\in F\left(w,s\right)+\Ncal_{\Zcal\times\Rset^m}\left(w\right)\enspace,
\end{align}
where 
\begin{align}
F\left(w,s\right):=\begin{bmatrix}
	\nabla_zJ(z)+\nabla_zG(z,s)^{\Trans}\mu\\
	G(z,s)
\end{bmatrix}\enspace,&~w=\begin{bmatrix}
z\\ \mu
\end{bmatrix}\enspace.
\end{align}
\looseness-1In order to analyse the behaviour of the KKT points of\ \eqref{eq:prm_nlp_1} as the parameter $s_k$ evolves over time, the generalised equation\ \eqref{eq:ge_1} should satisfy some regularity assumptions.~This is captured by the\ \textit{strong regularity} concept \cite{robin1980,zav2010}.
\begin{definition}[Strong regularity,\cite{robin1980}]
Let $\Omega$ be a compact convex set in $\Rset^n$ and $f:\Rset^n\rightarrow\Rset^n$ a differentiable mapping.~A generalised equation $0\in f(x)+\Ncal_\Omega\left(x\right)$ is said to be\ \textit{strongly regular at a solution} $x^\ast\in\Omega$ if there exists radii $\eta>0$ and $\kappa>0$ such that for all $r\in\Bcal\left(0,\eta\right)$, there exits a unique $x_r\in\Bcal\left(x^\ast,\kappa\right)$ such that
\begin{align}
\label{eq:strg_reg}
r\in f(x^\ast)+{\nabla}f(x^\ast)\left(x_r-x^\ast\right)+\Ncal_\Omega\left(x_r\right)\enspace,
\end{align}
and the inverse mapping $r\mapsto x_r$ from $\Bcal\left(0,\eta\right)$ to $\Bcal\left(x^\ast,\kappa\right)$ is Lipschitz continuous. 
\end{definition}
\begin{rk}
Note that strong regularity incorporates active-set changes in its definition, as the normal cone is taken at $x_r$ in Eq.\ \eqref{eq:strg_reg}.~The set of active constraints at $x_r$ may be different from the one at $x^\ast$.~Nevertheless, Lipschitz continuity of the solution is guaranteed.
\end{rk}
\begin{rk}
As the constraint set $\Zcal$ in\ \eqref{eq:prm_nlp_1} is polyhedral, it can be shown that strong regularity of a KKT point of\ \eqref{eq:prm_nlp_1} is equivalent to linear independence constraints qualification and strong second-order optimality\ \cite{don1996}, which are standard assumptions in nonlinear programming.
\end{rk}
As parameter $s_k$ changes in time, strong regularity is assumed at every time instant $k$.
\begin{ass}
\label{ass:strg_reg}
For all parameters $s_k\in\Scal$ and associated solutions $w^\ast_k$, the generalised equation\ \eqref{eq:ge_1} is strongly regular at $w^\ast_k$.
\end{ass}
From the strong regularity Assumption\ \ref{ass:strg_reg}, it can be proven that the non-smooth manifold formed by the solutions to the parametric program\ \eqref{eq:prm_nlp_1} is locally Lipschitz continuous.~The first step to achieve this fundamental property is the following Theorem proven in\ \cite{robin1980}.
\begin{theo}
\label{theo:strg_reg_th}
There exists radii $\delta_A>0$ and $r_A>0$ such that for all $k\in\Nset$, for all $s\in\Bcal\left(s_k,r_A\right)$, there exists a unique $w^\ast\left(s\right)\in\Bcal\left(w^\ast_k,\delta_A\right)$ such that
\begin{align}
0\in F\left(w^\ast\left(s\right),s\right)+\Ncal_{\Zcal\times\Rset^m}\left(w^\ast\left(s\right)\right)
\end{align}
and for all $s,s'\in\Bcal\left(s_k,r_A\right)$, 
\begin{align} 
\label{eq:robin}
\left\|w^\ast\left(s\right)-w^\ast\left(s'\right)\right\|_2\leq\lambda_A\left\|F\left(w^\ast\left(s'\right),s\right)-F\left(w^\ast\left(s'\right),s'\right)\right\|_2\enspace,
\end{align}
where $\lambda_A$ is a Lipschitz constant associated with the strong regularity mapping of \eqref{eq:ge_1}.
\end{theo}
\begin{rk}
Theorem\ \ref{theo:strg_reg_th} is actually a refinement of Theorem $2.1$ in\ \cite{robin1980}, as the radii $\delta_A$ and $r_A$ are assumed not to depend on the parameter $s_k\in\Scal$.
\end{rk}
Relation\ \eqref{eq:robin} does not exactly correspond to a Lipschitz property.~Yet, this point is addressed by the following Lemma.
\begin{lem}
\label{lem:lip_map}
There exists $\lambda_F>0$ such that for all $w\in\Zcal\times\Rset^m$,
\begin{align}
\label{eq:lip_F}
\forall s,s'\in\Scal, \left\|F(w,s)-F(w,s')\right\|_2\leq\lambda_F\left\|s-s'\right\|_2\enspace.
\end{align}
\end{lem}
\begin{proof}
Let $w\in\Zcal\times\Rset^m$ and $s,s'\in\Scal$.
\begin{align}
\nonumber 
F\left(w,s\right)-F\left(w,s'\right)&=\begin{bmatrix}
\left(\nabla_zG\left(z,s\right)-\nabla_zG\left(z,s'\right)\right)^{\Trans}\mu\\
G\left(z,s\right)-G\left(z,s'\right)\end{bmatrix}\nonumber\\
&=\begin{bmatrix}
0\\
T_1\left(s-s'\right)\\
\ldots\\
T_P\left(s-s'\right)
\end{bmatrix}\enspace.\nonumber
\end{align}
Hence,\ \eqref{eq:lip_F} holds with 
\begin{align}
\nonumber
\lambda_F=P\cdot\max\left\{\left\|T_1\right\|_2,\ldots,\left\|T_P\right\|_2\right\}\enspace.
\end{align} 
\end{proof}
\looseness-1Algorithm\ \ref{algo:splt_trck} tracks the non-smooth solution manifold by traveling from neighbourhood to neighbourhood, where Lipschitz continuity of the primal-dual solution holds.~Such tracking procedures have been analysed thoroughly in the unconstrained case by\ \cite{diehl2005} for a Newton-type method, in the constrained case by\ \cite{zav2010} for an augmented Lagrangian approach and in\ \cite{quoc2012} for an adjoint-based technique.~These previous tracking strategies are purely centralised second-order strategies and do not readily extend to solving NLPs in a distributed manner.~Our Algorithm\ \ref{algo:splt_trck} proposes a novel way of computing predictor steps along the solution manifold via a decomposition approach, which is tailored to convex constraint sets with closed-form proximal operators.~Such a class encompasses boxes, non-negative orthant, semi-definite cones and balls for instance.~The augmented Lagrangian framework is particularly attractive in this context, as it allows one to preserve `nice' constraints via partial penalisation.

\subsection{Convergence of the inner loop}
\label{subsec:prim_cv}
\looseness-1The primal loop of Algorithm\ \ref{algo:splt_trck} consists of alternating proximal gradient steps.~In general, for non-convex programs as they appear in NMPC, the convergence of such Gauss-Seidel type methods to critical points of the objective is not guaranteed even for smooth functions, as oscillatory behaviours may occur\ \cite{pow1973}.~Yet, some powerful convergence results on alternating minimisation techniques have been recently derived\ \cite{att2013}.~The key ingredients are coordinate-wise proximal regularisations and the KL property\ \cite{att2009}.~The former enforces a sufficient decrease in the objective at every iteration, while the latter models a local growth of the function around its critical points.~In our analysis, the convergence of the primal sequence generated by Algorithm\ \ref{algo:splt_trck} is of primary importance, since it comes with a sub-linear convergence rate estimate, which depends on the so-called Lojasiewicz exponent\ \cite{att2009} of the augmented Lagrangian function and does not assume that the optimal active set has been identified.~This last point is of primary importance in a parametric setting, as there are no guarantees that the active-set at $\bar{z}_k$ is the same as the one at $z^{\ast}_{k+1}$.

\looseness-1The following Theorem is a formulation of the KL property for a multivariate polynomial function over a box.~In this particular case, the Lojasiewicz exponent can be explicitly computed.~It is proven to be a simple function of the degree of the polynomial and its dimension.   

\begin{theo}
\label{th:kl_poly}
Let $L:\Rset^n\rightarrow\Rset$, $n\geq1$, be a polynomial function of degree $\deg\left(L\right)\geq2$.~Let $\Omega\subset\Rset^n$ be a non-trivial polyhedral set.~Assume that all restrictions of $L$ to faces of $\Omega$ that are not vertices, have degree larger than two.~Given $x^{\ast}$ a critical point of $L+\iota_\Omega$, there exists constants $\delta>0$ and $c>0$ such that for all $x\in\Bcal\left(x^{\ast},\delta\right)\cap\Omega$ and all $v\in\Ncal_\Omega\left(x\right)$,  
\begin{align}
\label{eq:kl_smal}
\left\|\nabla{}L\left(x\right)+v\right\|_2\geq c\left|L\left(x\right)-L\left(x^{\ast}\right)\right|^{\theta\left(\deg\left(L\right),n\right)}\enspace,
\end{align}
where 
\begin{align}
\label{eq:def_th}
\theta\left(d,n\right):=1-\displaystyle\frac{1}{d\left(3d-3\right)^{n-1}}\enspace.
\end{align}
\end{theo}
\begin{proof}
Let $x^{\ast}$ be a critical point of $L+\iota_\Omega$.~From\ \cite{att2009}, as $L+\iota_\Omega$ is a semi-algebraic function, there exists a radius $\delta'>0$, a constant $c'>0$ and a coefficient $\theta'\in\left(0,1\right)$ such that for all $x\in\Bcal\left(x^{\ast},\delta'\right)\cap\Omega$ and all $v\in\Ncal_\Omega\left(x\right)$,
\begin{align}
\label{eq:kl_1}
\left\|\nabla{}L\left(x\right)+v\right\|_2\geq{}c'\left|L\left(x\right)-L\left(x^{\ast}\right)\right|^{\theta'}\enspace.
\end{align}
Define $\theta'_f$ as the infimum of all $\theta'$ for which\ \eqref{eq:kl_1} is satisfied.~Our goal is to show that
\begin{align}  
\nonumber
\theta'_f\leq\theta\left(\deg\left(L\right),n\right)\enspace,
\end{align}
as it directly implies that\ \eqref{eq:kl_smal} is satisfied.~One can assume that $\theta'_f>0$, since for $\theta'_f=0$ the proof would be immediate.~For the sake of contradiction, assume that 
\begin{align}
\nonumber
\theta'_f>\theta\left(\deg\left(L\right),n\right)\enspace.
\end{align}
Hence, one can pick $\tilde{\theta}\in\left(\theta\left(\deg\left(L\right),n\right),\theta'_f\right)$ and $c''>0$, and construct a sequence $\left\{\left(x_n,v_n\right)\right\}$ satisfying for all $n\geq1$,
\begin{align}
\label{eq:contra}
\left\{
\begin{aligned}
&x_n\in\Bcal\left(x^{\ast},\displaystyle\frac{1}{n}\right)\cap\Omega,~v_n\in\Ncal_\Omega\left(x_n\right)\\
&\left\|\nabla{}L\left(x_n\right)+v_n\right\|_2<c''\left|L\left(x_n\right)-L\left(x^{\ast}\right)\right|^{\tilde{\theta}}
\end{aligned}\enspace.
\right.
\end{align}
Without loss of generality, one can find a face $\Fcal$ of $\Omega$, which is not a vertex and contains $x^{\ast}$, and a subsequence $\left\{x_{n_k}\right\}$ such that 
\begin{align}
\nonumber
x_{n_k}\in\ri\Fcal\enspace,
\end{align}
for $k$ large enough and satisfying\ \eqref{eq:contra}.~Moreover, for all $x\in\ri\Fcal$, there exists $p\in\Rset^{d_\Fcal}$
\begin{align}  
x=x^{\ast}+Zp\enspace,\nonumber
\end{align}
where $Z\in\Rset^{n\times{}d_\Fcal}$ is a full column-rank matrix, with $d_\Fcal$ the dimension of the affine hull of $\Fcal$.~As the face $\Fcal$ is not a vertex, $d_\Fcal\geq1$.~Subsequently, one can define a polynomial function $L^{\ast}:\Rset^{d_\Fcal}\rightarrow\Rset$ as follows
\begin{align}
\nonumber
L^{\ast}\left(p\right):=L\left(x^{\ast}+Zp\right)\enspace.
\end{align}
From the results of\ \cite{acunto2005} (no matter whether $0$ is a critical point of $L^{\ast}$ or not, see Remark $3.2$ in\ \cite{att2010}), as $\deg\left(L^{\ast}\right)\geq2$ by assumption, there exists a radius $\delta^{\ast}>0$ and a constant $c^{\ast}>0$ such that for all $p\in\Bcal\left(0,\delta^{\ast}\right)$    
\begin{align}
\nonumber
\left\|\nabla{}L^{\ast}\left(p\right)\right\|_2\geq{}c^{\ast}\left|L^{\ast}\left(p\right)-L^{\ast}\left(0\right)\right|^{\theta\left(\deg\left(L^{\ast}\right),d_\Fcal\right)}\enspace.
\end{align}
However, $\deg\left(L^{\ast}\right)\leq\deg\left(L\right)$ and $d_\Fcal\leq{}n$, which implies that
\begin{align}
\theta\left(\deg\left(L^{\ast}\right),d_\Fcal\right)\leq\theta\left(\deg\left(L\right),n\right)\enspace.
\end{align}
As $L^{\ast}$ is a continuous function, the radius $\delta^{\ast}$ can always be chosen such that 
\begin{align}
 \nonumber
 \left|L^{\ast}\left(p\right)-L^{\ast}\left(0\right)\right|<1\enspace.
\end{align}
This implies that for all $p\in\Bcal\left(0,\delta^{\ast}\right)$, 
\begin{align}
\nonumber
\left\|\nabla{}L^{\ast}\left(p\right)\right\|_2\geq{}c^{\ast}\left|L^{\ast}\left(p\right)-L^{\ast}\left(0\right)\right|^{\theta\left(\deg\left(L\right),n\right)}\enspace.\nonumber
\end{align}
Hence, there exists $K\geq1$ such that for all $k\geq{}K$,
\begin{align}
\nonumber
\left\|\nabla{}L\left(x_{n_k}\right)+v_{n_k}\right\|_2&\geq\frac{1}{\left\|Z\right\|_2}\left\|Z^{\Trans}\left(\nabla{}L\left(x_{n_k}\right)+v_{n_k}\right)\right\|_2\nonumber\\
&\geq\frac{1}{\left\|Z\right\|_2}\left\|Z^{\Trans}\left(\nabla{}L\left(x^{\ast}+Zp_{n_k}\right)+v_{n_k}\right)\right\|_2\nonumber\\
&\geq\frac{1}{\left\|Z\right\|_2}\left\|\nabla{}L^{\ast}\left(p_{n_k}\right)\right\|_2\nonumber\\
&\geq\frac{c^{\ast}}{\left\|Z\right\|_2}\left|L\left(x_{n_k}\right)-L\left(x^{\ast}\right)\right|^{\theta\left(\deg\left(L\right),n\right)}\enspace.\nonumber
\end{align}
The third inequality follows from $Z^{\Trans}v_{n_k}=0$, as $v_{n_k}$ is in the normal cone to $\Fcal$.~However, since $c''$ can be chosen equal to $\nicefrac{c^{\ast}}{\left\|Z\right\|_2}$ as $\Omega$ has finitely many faces, the above implies that 
\begin{align}
\nonumber
\left|L\left(x_{n_k}\right)-L\left(x^{\ast}\right)\right|^{\theta\left(\deg\left(L\right),n\right)}<\left|L\left(x_{n_k}\right)-L\left(x^{\ast}\right)\right|^{\tilde{\theta}}\enspace.
\end{align}
This leads to a contradiction for $k$ large enough so that $\left|L\left(x_{n_k}\right)-L\left(x^{\ast}\right)\right|<1$, as $\tilde{\theta}>\theta\left(\deg\left(L\right),n\right)$ by assumption.
\end{proof}
\begin{cor}
\label{cor:kl_al}
\looseness-1Given $\mu\in\Rset^{m+q}$, $s\in\Scal$ and $\rho>0$, $L\left(\cdot,\mu,s\right)+\iota_\Zcal$ satisfies inequality\ \eqref{eq:kl_smal} around all its critical points with radius $\delta>0$ and constant $c>0$, where $L\left(\cdot,\mu,s\right)$ is the augmented Lagrangian function defined in\ \eqref{eq:al_1}.
\end{cor}
\begin{proof}
This is an immediate consequence of Theorem\ \ref{th:kl_poly} and Assumption\ \ref{ass:ass_1}.
\end{proof}
In order to guarantee convergence of the primal loop of Algorithm\ \ref{algo:splt_trck} via Theorem $2.9$ in\ \cite{att2013}, two ingredients are needed: a\ \textit{sufficient decrease} property and a\ \textit{relative error} condition.~From the sufficient decrease, convergence of the series $\sum\left\|z^{(l+1)}-z^{(l)}\right\|_2^2$ is readily deduced.~By combining the relative error condition and the KL property, this can be turned into convergence of the series $\sum\left\|z^{(l+1)}-z^{(l)}\right\|_2$, ensuring convergence of the sequence $\left\{z^{(l)}\right\}$ via a Cauchy sequence argument\ \cite{att2013}.
\begin{lem}[\label{lem:suff_dec}Primal sufficient decrease]
\looseness-1For all $l\geq1$, $\mu\in\Rset^{m+q}$, $s\in\Scal$ and $\rho>0$,
\begin{align}
\label{eq:suff_dec}
L_\rho\left(z^{(l+1)},\mu,s\right)+&\iota_\Zcal\left(z^{(l+1)}\right)+\displaystyle\frac{\underline{\alpha}}{2}\left\|z^{(l+1)}-z^{(l)}\right\|_2^2\\
&\leq L_\rho\left(z^{(l)},\mu,s\right)+\iota_\Zcal\left(z^{(l)}\right)\enspace,\nonumber
\end{align}
where $\underline{\alpha}:=\min\left\{\alpha_i~\big|~i\in\left\{1,\ldots,P\right\}\right\}$.
\end{lem}
\begin{IEEEproof}
We first need to show that the backtracking procedure described in Algorithm\ \ref{algo:loc_bck} terminates.~This is an almost direct consequence of the Lipschitz continuity of the gradient of $L^{(i)}_{\rho,\mu,s}$ for $i\in\left\{1,\ldots,P\right\}$, as the augmented Lagrangian is twice continuously differentiable and $\Zcal$ is compact.~From Lemma\ \ref{lem:desc_lem}, it follows that for all $i\in\left\{1,\ldots,P\right\}$ and $l\geq1$,
\begin{align}
&L^{(i,l)}_{\rho,\mu,s}\left(z_i^{(l+1)}\right)\leq L^{(i,l)}_{\rho,\mu,s}\left(z_i^{(l)}\right)\nonumber\\
&+{\nabla}L^{(i,l)}_{\rho,\mu,s}\left(z_i^{(l)}\right)^{\Trans}\left(z_i^{(l+1)}-z_i^{(l)}\right)+\displaystyle\frac{\lambda_i}{2}\left\|z_i^{(l+1)}-z_i^{(l)}\right\|_2^2\enspace,\nonumber
\end{align}
where $\lambda_i$ is a Lipschitz constant of ${\nabla}L^{(i)}_{\rho,\mu,s}$.~By taking 
\begin{align}
c_i>\lambda_i+\alpha_i\enspace,\nonumber
\end{align} 
which is satisfied at some point in the\ \texttt{while} loop of Algorithm \ref{algo:loc_bck}, since $\beta>1$, one gets 
\begin{align}
L^{(i,l)}_{\rho,\mu,s}&\left(z_i^{(l+1)}\right)+\displaystyle\frac{\alpha_i}{2}\left\|z_i^{(l+1)}-z_i^{(l)}\right\|_2^2\leq L^{(i,l)}_{\rho,\mu,s}\left(z_i^{(l)}\right)\nonumber\\
&+{\nabla}L^{(i,l)}_{\rho,\mu,s}\left(z_i^{(l)}\right)^{\Trans}\left(z_i^{(l+1)}-z_i^{(l)}\right)+\displaystyle\frac{c_i}{2}\left\|z_i^{(l+1)}-z_i^{(l)}\right\|_2^2\enspace,\nonumber
\end{align}
which is exactly the termination criterion of the\ \texttt{while} loop in Algorithm\ \ref{algo:loc_bck}.~Moreover,
\begin{align}
z_i^{(l+1)}&=\prox^{\iota_{\Zcal_i}}_{c_i}\left(z_i^{(l)}-\displaystyle\frac{1}{c_i}{\nabla}L^{(i,l)}_{\rho,\mu,s}\left(z_i^{(l)}\right)\right)\nonumber\\
&=\argmin_{x\in\Zcal_i} \displaystyle\frac{c_i}{2}\left\|x-\left(z_i^{(l)}-\displaystyle\frac{1}{c_i}{\nabla}L^{(i,l)}_{\rho,\mu,s}\left(z_i^{(l)}\right)\right)\right\|_2^2\nonumber\\
&=\argmin_{x\in\Zcal_i} {\nabla}L^{(i,l)}_{\rho,\mu,s}\left(z_i^{(l)}\right)^{\Trans}\left(x-z_i^{(l)}\right)+\displaystyle\frac{c_i}{2}\left\|x-z_i^{(l)}\right\|_2^2\enspace.\nonumber
\end{align}
Hence
\begin{align}
{\nabla}L^{(i,l)}_{\rho,\mu,s}\left(z_i^{(l)}\right)^{\Trans}\left(z_i^{(l+1)}-z_i^{(l)}\right)+\displaystyle\frac{c_i}{2}\left\|z_i^{(l+1)}-z_i^{(l)}\right\|_2^2\leq 0\enspace,\nonumber
\end{align}
which implies that
\begin{align}
\label{eq:suff_dec_i}
L^{(i,l)}_{\rho,\mu,s}\left(z_i^{(l+1)}\right)+&\iota_{\Zcal_i}\left(z_i^{(l+1)}\right)+\displaystyle\frac{\alpha_i}{2}\left\|z_i^{(l+1)}-z_i^{(l)}\right\|_2^2\\
&\leq L^{(i)}_{\rho,\mu,s}\left(z_i^{(l)}\right)+\iota_{\Zcal_i}\left(z_i^{(l)}\right)\enspace,\nonumber
\end{align}
\looseness-1as $z_i^{(l+1)},z_i^{(l)}\in\Zcal_i$.~By summing inequalities\ \eqref{eq:suff_dec_i} for all $i\in\left\{1,\ldots,P\right\}$, one obtains the sufficient decrease property\ \eqref{eq:suff_dec}.
\end{IEEEproof}

\begin{lem}[\label{lem:rel_err}Relative error condition]
For all $\mu\in\Rset^{m+q}$, $s\in\Scal$ and $\rho>0$, there exists $\gamma\left(\mu,\rho,s\right)>0$ such that 
\begin{align}
\label{eq:rel_err} 
&\exists v^{(l+1)}\in\Ncal_\Zcal\left(z^{(l+1)}\right),\\
&\left\|\nabla_zL_\rho\left(z^{(l+1)},\mu,s\right)+v^{(l+1)}\right\|_2\leq\gamma\left(\rho,\mu,s\right)\left\|z^{(l+1)}-z^{(l)}\right\|_2\enspace.\nonumber
\end{align}
\end{lem}
\begin{IEEEproof}
From the definition of $z_i^{(l+1)}$ as a proximal iterate, we have that for all $i\in\left\{1,\ldots,P\right\}$, 
\begin{align}
\exists v_i^{(l+1)}\in&~\Ncal_{\Zcal_i}\left(z_i^{(l+1)}\right),\nonumber\\
&0={\nabla}L^{(i,l)}_{\rho,\mu,s}\left(z_i^{(l)}\right)+c_i\left(z_i^{(l+1)}-z_i^{(l)}\right)+v_i^{(l+1)}\enspace.\nonumber
\end{align}
Hence 
\begin{align}
0={\nabla}L^{(i,l)}_{\rho,\mu,s}\left(z_i^{(l+1)}\right)&+{\nabla}L^{(i,l)}_{\rho,\mu,s}\left(z_i^{(l)}\right)-{\nabla}L^{(i,l)}_{\rho,\mu,s}\left(z_i^{(l+1)}\right)\nonumber\\
&+c_i\left(z_i^{(l+1)}-z_i^{(l)}\right)+v_i^{(l+1)}\enspace,
\end{align}
and from the Lipschitz continuity of ${\nabla}L^{(i)}_{\rho,\mu,s}$, one immediately obtains
\begin{align}
\left\|v_i^{(l+1)}+{\nabla}L^{(i,l)}_{\rho,\mu,s}\left(z_i^{(l+1)}\right)\right\|_2\leq\left(\lambda_i+c_i\right)\left\|z_i^{(l+1)}-z_i^{(l)}\right\|_2\enspace.\nonumber
\end{align}
Let $v:=\left(v_1^{\Trans},\ldots,v_P^{\Trans}\right)^{\Trans}$. It then follows that
\begin{align}
&\left\|v^{(l+1)}+{\nabla}L_{\rho,\mu,s}\left(z^{(l+1)}\right)\right\|_2\leq\nonumber\\
&\sum_{i=1}^P\left\|v_i^{(l+1)}+{\nabla}L_{\rho,\mu,s}^{(i,l)}\left(z_i^{(l+1)}\right)\right\|_2\nonumber\\
&~~~~~~~~~~~~+\left\|\nabla_iL_{\rho,\mu,s}\left(z^{(l+1)}\right)-{\nabla}L_{\rho,\mu,s}^{(i,l)}\left(z_i^{(l+1)}\right)\right\|_2\nonumber\\
&\leq\sum_{i=1}^P\left(\lambda_i+c_i\right)\left\|z_i^{(l+1)}-z_i^{(l)}\right\|_2+\lambda\left\|z^{(l+1)}-z^{(l)}\right\|_2\nonumber\\
&\leq\left(\sum_{i=1}^P\left(\lambda_i+c_i+\lambda\right)\right)\left\|z^{(l+1)}-z^{(l)}\right\|_2\enspace,\nonumber
\end{align}
where $\lambda>0$ is a Lipschitz constant of $\nabla{}L_{\rho,\mu,s}$.~This last inequality yields the relative error condition\ \eqref{eq:rel_err} by taking 
\begin{align}
\nonumber
\gamma\left(\rho,\mu,s\right):=\sum_{i=1}^P\left(\lambda_i+c_i+\lambda\right)\enspace.
\end{align}
\end{IEEEproof}

\begin{theo}[\label{th:prim_cv}Global convergence of the primal sequence]
\looseness-1Taking $M=\infty$ in Algorithm\ \ref{algo:splt_trck}, the primal sequence $\left\{z^{(l)}\right\}$ converges to a critical point $z^\infty\left(\bar{\mu}\left(s_k\right),s_{k+1}\right)$ of $L_\rho\left(\cdot,\bar{\mu}\left(s_k\right),s_{k+1}\right)+\iota_\Zcal$. 
\end{theo}
\begin{IEEEproof}
\looseness-1By Theorem\ \ref{th:kl_poly} and Corollary\ \ref{cor:kl_al}, the function $L_\rho\left(\cdot,\bar{\mu}\left(s_k\right),s_{k+1}\right)+\iota_\Zcal$ satisfies the KL property.~Moreover, sufficient decrease is guaranteed by Lemma\ \ref{lem:suff_dec} along with a relative error condition in Lemma\ \ref{lem:rel_err}.~As $L_\rho\left(\cdot,\bar{\mu}\left(s_k\right),s_{k+1}\right)$ is continuous and $\Zcal$ compact, global convergence of the sequence $\left\{z^{(l)}\right\}$ to a critical point of $L_\rho\left(\cdot,\bar{\mu}\left(s_k\right),s_{k+1}\right)+\iota_\Zcal$ is a direct consequence of Theorem $2.9$ in\ \cite{att2013}.
\end{IEEEproof}

\subsection{Convergence rate of the primal loop}
\label{subsec:prim_cv_rate}
\looseness-1The results of\ \cite{att2009} and\ \cite{att2013} provide an asymptotic convergence rate estimate for the proximal alternating loop in Algorithm\ \ref{algo:splt_trck}.~The convergence rate depends on the Lojasiewicz exponent
\begin{align}
\nonumber
\theta\left(d_L,n_z\right)
\end{align}
defined in\ \eqref{eq:def_th}, which only depends on the dimension of NLP\ \eqref{eq:prm_nlp_1} and the degree of the polynomial functions involved in it.~This is an important point in our analysis, as $\bar{\mu}_k$ and $s_k$ are updated at every time instant.

\begin{lem}[\label{lem:cv_rate}Asymptotic convergence rate estimate]
There exists a constant $C>0$ such that, assuming $\bar{z}\left(s_k\right)\in\Bcal\left(0,\delta\right)$, where $\delta$ has been defined in Corrolary\ \ref{cor:kl_al},
\begin{align}
\label{eq:cv_rate}
&\left\|\bar{z}\left(s_{k+1}\right)-z^{\infty}\left(\bar{\mu}_k,s_{k+1}\right)\right\|_2\leq\\\nonumber
&~~~~~~CM^{-\psi(d_L,n_z)}\left\|\bar{z}\left(s_k\right)-z^\infty\left(\bar{\mu}_k,s_{k+1}\right)\right\|_2\enspace,\nonumber
\end{align}
where
\begin{align}
\psi\left(d,n\right):=\displaystyle\frac{1}{d\left(3d-3\right)^{n-1}-2}\enspace,
\end{align}
for $d,n\geq2$.
\end{lem}
\begin{IEEEproof}
As $n_z\geq2$ and $d_L\geq2$,
\begin{align}
\nonumber
\theta\left(d_L,n_z\right)\in\left(\displaystyle\frac{1}{2},1\right)\enspace.
\end{align}
Inequality\ \eqref{eq:cv_rate} is then a direct consequence of Theorem $2$ in\ \cite{att2009} as the initial primal iterate is $\bar{z}\left(s_k\right)$. 
\end{IEEEproof}
\begin{rk}
\looseness-1The R-convergence rate estimate of Lemma\ \ref{lem:cv_rate} shows that the convergence of the primal sequence $\left\{z^{(l)}\right\}$ is theoretically sub-linear.~However, reasonable performance can be observed in practice.~Moreover, in this paper, the convergence rate is used only for a theoretical purpose.  
\end{rk}
\section{Contraction of the primal-dual sequence}
\label{sec:con_prop}
\looseness-1Algorithm\ \ref{algo:splt_trck} is a truncated scheme both in the primal and dual space, as only $M$ primal proximal iterations are applied, which are followed by a single dual update.~By using warm-starting, it is designed to track the non-smooth solution manifold of the NMPC program.~At a given time instant $k$, the primal-dual solution $\bar{w}_k$ is suboptimal.~Thus, a natural question is whether the sub-optimality gap remains stable, as the parameter $s_k$ varies over time, that is if the sub-optimal iterate remains close to the KKT manifold, or converges to it.~Intuitively, one can guess that if $s_k$ evolves slowly and the number of primal iterations $M$ is large enough, stability of the sub-optimality error is expected.~This section provides a formal statement about the sub-optimality gap and demonstrates that its evolution is governed by the penalty parameter $\rho$, the number of primal iterations $M$ and the magnitude of the parameter difference $s_{k+1}-s_k$, which need to be carefully chosen according to the results provided later in the paper.

\subsection{Existence and uniqueness of critical points}
\label{subsec:uni_kkt} 
\looseness-1As the overall objective is to analyse the stability of the sub-optimality error $\left\|\bar{w}_k-w^{\ast}_k\right\|_2$, a unique critical point $w^\ast_k$ should be defined at every time instant $k$.~This is one of the roles of strong regularity.~Given a critical point $w_k^\ast$ for problem \eqref{eq:prm_nlp_1} at $s_k$, its strong regularity (Assumption \ref{ass:strg_reg}) implies that there exists a unique critical point for problem\ \eqref{eq:prm_nlp_1} at $s_{k+1}$, assuming $\left\|s_{k+1}-s_k\right\|_2$ is small enough.
\begin{ass}
\label{ass:prm_diff}
For all $k\geq 0$, $\left\|s_{k+1}-s_k\right\|_2\leq r_A$. 
\end{ass}
\begin{rk}
\looseness-1In an NMPC setting, this assumption is satisfied if the sampling frequency is fast enough compared to the system's dynamics.
\end{rk}
\begin{lem}
For all $k\geq 0$ and $s_k\in\Scal$, given $w^\ast_k$ such that
\begin{align}
0\in F\left(w^\ast_k,s_k\right)+\Ncal_{\Zcal\times\Rset^m}\left(w^\ast_k\right)\enspace,\nonumber
\end{align}
there exists a unique $w^\ast_{k+1}\in\Bcal\left(w^\ast_k,\delta_A\right)$ such that 
\begin{align}
0\in F\left(w^\ast_{k+1},s_{k+1}\right)+\Ncal_{\Zcal\times\Rset^m}\left(w^\ast_{k+1}\right)\enspace.\nonumber
\end{align}
\end{lem}
\begin{IEEEproof}
\looseness-1This is an immediate consequence of Assumption\ \ref{ass:prm_diff} and strong regularity of $w^\ast_k$ for all $k\geq 0$. 
\end{IEEEproof}

\subsection{An auxiliary generalised equation}
\label{subsec:aux_ge}
\looseness-1In Algorithm\ \ref{algo:splt_trck}, the primal loop, which is initialised at $\bar{z}_k$, converges to $z^\infty\left(\bar{\mu}_k,s_{k+1}\right)$, a critical point of $L_\rho\left(\cdot,\bar{\mu}_k,s_{k+1}\right)+\iota_\Zcal\left(\cdot\right)$, by Theorem\ \ref{th:prim_cv} in Section\ \ref{sec:theo_tools}.~The following generalised equation characterises critical points of the augmented Lagrangian function $L_\rho\left(\cdot,\bar{\mu},s\right)+\iota_\Zcal\left(\cdot\right)$ in a primal-dual manner:
\begin{align}
\label{eq:ge_2}
0\in H_\rho\left(w,d_\rho\left(\bar{\mu}\right),s\right)+\Ncal_{\Zcal\times\Rset^m}\left(w\right)\enspace,
\end{align}
where, given $\mu^\ast_k$, one defines $d_\rho\left(\bar{\mu}\right):=\left(\bar{\mu}-\mu^\ast_k\right)/\rho$ and 
\begin{align}
H_\rho\left(w,d_\rho\left(\bar{\mu}\right),s\right):=\begin{bmatrix}
\nabla_zJ\left(z\right)+\nabla_zG\left(z,s\right)^{\Trans}\mu\\
G\left(z,s\right)+d_\rho\left(\bar{\mu}\right)+\displaystyle\frac{\mu^{\ast}_k-\mu}{\rho}
\end{bmatrix}\enspace.\nonumber
\end{align}
\begin{lem}
Let $\bar{\mu}\in\Rset^m$, $\rho>0$ and $s\in\Scal$.~The primal point $z^\ast(\bar{\mu},s)$ is a critical point of $L_{\rho}(\cdot,\bar{\mu},s)+\iota_\Zcal(\cdot)$ if and only if the primal-dual point 
\begin{align}
w^\ast\left(d_\rho\left(\bar{\mu}\right),s\right)=\begin{bmatrix}
z^\ast\left(\bar{\mu},s\right)\\
\bar{\mu}+\rho G\left(z^\ast\left(\bar{\mu},s\right),s\right)
\end{bmatrix}\nonumber
\end{align}
satisfies\ \eqref{eq:ge_2}.
\end{lem}
\begin{IEEEproof}
The necessary condition is clear.~To prove the sufficient condition, assume that $w^\ast\left(d_\rho\left(\bar{\mu}\right),s\right)=\left(z^\ast\left(d_\rho\left(\bar{\mu}\right),s\right)^{\Trans},\mu^\ast\left(d_\rho\left(\bar{\mu}\right),s\right)^{\Trans}\right)^{\Trans}$ satisfies \eqref{eq:ge_2}.~The second half of \eqref{eq:ge_2} implies that $\mu^\ast\left(d_\rho\left(\bar{\mu}\right),s\right)=\bar{\mu}+\rho G\left(z^\ast\left(d_\rho\left(\bar{\mu}\right),s\right),s\right)$.~Putting this expression in the first part of \eqref{eq:ge_2}, one obtains that $z^\ast\left(d_\rho\left(\bar{\mu}\right),s\right)$ is a critical point of $L_\rho\left(\cdot,\bar{\mu},s\right)+\iota_\Zcal\left(\cdot\right)$.
\end{IEEEproof}
\looseness-1In the sequel, a primal-dual point satisfying\ \eqref{eq:ge_2} is denoted by $w^{\ast}\left(d_\rho\left(\bar{\mu}\right),s\right)$ or $w^{\ast}\left(\bar{\mu},s\right)$ without distinction. 

\looseness-1As $z^\infty\left(\bar{\mu}_k,s_{k+1}\right)$ is a critical point of $L_\rho\left(\cdot,\bar{\mu}_k,s_{k+1}\right)+\iota_\Zcal\left(\cdot\right)$, one can define 
\begin{align}
\label{eq:def_w_inf}
w^\infty\left(d_\rho(\bar{\mu}_k),s_{k+1}\right):=\begin{bmatrix}
z^\infty\left(\bar{\mu}_k,s_{k+1}\right)\\
\bar{\mu}_k+\rho G\left(z^\infty(\bar{\mu}_k,s_{k+1}),s_{k+1}\right)
\end{bmatrix}\enspace,
\end{align}
which satisfies\ \eqref{eq:ge_2}.~Note that the generalised equation\ \eqref{eq:ge_2} is parametric in $s$ and $d_\rho(\cdot)$, which represents a normalised distance between a dual variable and an optimal dual variable at time $k$.~Assuming that the penalty parameter $\rho$ is well-chosen, the generalised equation\ \eqref{eq:ge_2} can be proven to be strongly regular at a given solution.
\begin{lem}[Strong regularity of\ \eqref{eq:ge_2}]
There exists $\tilde{\rho}>0$ such that for all $\rho>\tilde{\rho}$ and $k\geq0$,\ \eqref{eq:ge_2} is strongly regular at $w^\ast_k=w^\ast\left(0,s_k\right)$. 
\end{lem}
\begin{IEEEproof}
As $\Zcal$ is polyhedral, this follows from the reduction procedure described in\ \cite{robin1980}, the arguments developed in Proposition $2.4$ in\ \cite{bert1982} and strong regularity of \eqref{eq:ge_1} for all $k\geq0$.
\end{IEEEproof}
\begin{ass}
\label{ass:rho}
The penalty parameter satisfies $\rho>\tilde{\rho}$.
\end{ass}
From the strong regularity of\ \eqref{eq:ge_2} at $w^\ast_k$, using Theorem $2.1$ in\ \cite{robin1980}, one obtains the following local Lipschitz property of a solution $w\left(\cdot\right)$ to \eqref{eq:ge_2}. 
\begin{lem}
\label{lem:str_reg_2}
There exists radii $\delta_B>0$, $r_B>0$ and $q_B>0$ such that for all $k\in\Nset$,
\begin{align}
\forall d\in\Bcal\left(0,q_B\right),&\forall s\in\Bcal\left(s_k,r_B\right),\exists!w^{\ast}(d,s)\in\Bcal\left(w^{\ast}_k,\delta_B\right),\nonumber\\
&0\in H_\rho(w^{\ast}(d,s),d,s)+\Ncal_{\Zcal\times\Rset^m}\left(w^{\ast}(d,s)\right)\nonumber
\end{align}
and for all $d,d'\in\Bcal\left(0,q_B\right)$ and all $s,s'\in\Bcal\left(s_k,r_B\right)$,
\begin{align}
&\left\|w^{\ast}(d,s)-w^{\ast}(d',s')\right\|_2\leq\nonumber\\
&~~~~~~~~~\lambda_B\left\|H_\rho\left(w^\ast(d',s'),d,s\right)-H_\rho\left(w^\ast(d',s'),d',s'\right)\right\|_2\enspace,\nonumber
\end{align}
where $\lambda_B>0$ is a Lipschitz constant associated with\ \eqref{eq:ge_2}.
\end{lem}
Note that, given $w\in\Zcal\times\Rset^m$, $d,d'\in\Rset^m$ and $s,s'\in\Scal$, one can write
\begin{align}
H_\rho\left(w,d,s\right)-H_\rho\left(w,d',s'\right)=F(w,s)-F(w,s')+\begin{bmatrix}
0\\
d-d'
\end{bmatrix}\enspace,\nonumber
\end{align}
which, from Lemma\ \ref{lem:lip_map}, implies the following Lemma.
\begin{lem}
\label{lem:glob_lip_H}
There exists $\lambda_H>0$ such that for all $w\in\Zcal\times\Rset^m$, for all $d,d'\in\Rset^m$ and all $s,s'\in\Rset^m$, 
\begin{align}
\label{eq:glob_lip_H}
\left\|H_\rho\left(w,d,s\right)-H_\rho\left(w,d',s'\right)\right\|_2\leq\lambda_H\left\|\begin{pmatrix}
d\\ s
\end{pmatrix}-\begin{pmatrix}
d'\\ s'
\end{pmatrix}\right\|_2\enspace.
\end{align}
\end{lem}
\begin{IEEEproof}
After straightforward calculations, one obtains the Lipschitz property with  
\begin{align}
\lambda_H:=\sqrt{\max\big\{\lambda_F^2,1\big\}+\lambda_F}\enspace.\nonumber
\end{align}
\end{IEEEproof}

\subsection{Derivation of the contraction inequality}
\label{subsec:stab_ana}

In this paragraph, it is proven that under some conditions, which are made explicit in the sequel, the optimality tracking error $\left\|\bar{w}_k-w^\ast_k\right\|_2$ of Algorithm\ \ref{algo:splt_trck} remains within a pre-specified bound if the parameter $s_k$ varies sufficiently slowly over time.

First, note that given a sub-optimal primal-dual solution $\bar{w}_{k+1}$ and a critical point $w^\ast_{k+1}$, 
\begin{align}
\label{eq:stb_ineq}
\left\|\bar{w}_{k+1}-w^{\ast}_{k+1}\right\|_2\leq&\left\|\bar{w}_{k+1}-w^\infty\left(d_\rho\left(\bar{\mu}_k\right),s_{k+1}\right)\right\|_2\\
&+\left\|w^\infty\left(d_\rho\left(\bar{\mu}_k\right),s_{k+1}\right)-w^\ast_{k+1}\right\|_2\enspace,\nonumber
\end{align}
where $w^\infty\left(d_\rho(\bar{\mu}_k),s_{k+1}\right)$ has been defined in\ \eqref{eq:def_w_inf}.~The analysis then consists in bounding the two right hand side terms in\ \eqref{eq:stb_ineq}.~The first one can be upper-bounded by applying strong regularity of\ \eqref{eq:ge_2} and the second one using the convergence rate of the primal loop in Algorithm\ \ref{algo:splt_trck}. 
\begin{lem}
\label{lem:fir_term}
If $\left\|s_{k+1}-s_k\right\|_2$ satisfies 
\begin{align}
\label{eq:hypo_ds_1}
\left\|s_{k+1}-s_k\right\|_2<\min\left\{r_B,\displaystyle\frac{q_B\rho}{\lambda_A\lambda_F}\right\}\enspace,
\end{align}
where $r_B$ and $q_B$ have been defined in Lemma\ \ref{lem:str_reg_2}, and $\left\|\bar{w}_k-w^\ast_k\right\|_2<q_B\rho$, then,  
\begin{align}
&\left\|w^\infty\left(d_\rho\left(\bar{\mu}_k\right),s_{k+1}\right)-w^\ast_{k+1}\right\|_2\leq\nonumber\\
&~~~~~~~~~~~~~\displaystyle\frac{\lambda_B\lambda_H}{\rho}\big(\left\|\bar{w}_k-w^\ast_k\right\|_2+\lambda_A\lambda_F\left\|s_{k+1}-s_k\right\|_2\big)\enspace.\nonumber
\end{align}
\end{lem}
\begin{IEEEproof}
Note that $w^\ast_{k+1}$ can be rewritten $w^\ast_{k+1}=w^\ast\left(d_\rho\left(\mu^\ast_{k+1}\right),s_{k+1}\right)$, which is a solution to\ \eqref{eq:ge_2} at $s_{k+1}$.
\begin{align}
\left\|d_\rho\left(\mu^\ast_{k+1}\right)\right\|_2&=\displaystyle\frac{\left\|\mu^\ast_{k+1}-\mu^\ast_k\right\|_2}{\rho}\nonumber\\
&\leq\displaystyle\frac{\lambda_F\lambda_A}{\rho}\left\|s_{k+1}-s_k\right\|_2<q_B\enspace,\nonumber
\end{align}
by applying Theorem\ \ref{theo:strg_reg_th}, Lemma\ \ref{lem:lip_map} and from hypothesis\ \eqref{eq:hypo_ds_1}. Moreover,
\begin{align}
\left\|d_\rho(\bar{\mu}_k)\right\|_2&=\displaystyle\frac{\left\|\bar{\mu}_k-\mu^\ast_k\right\|_2}{\rho}\nonumber\\
&\leq\displaystyle\frac{\left\|\bar{w}_k-w^\ast_k\right\|_2}{\rho}<q_B\enspace.\nonumber
\end{align}
Now, as $\left\|s_{k+1}-s_k\right\|_2<r_B$ one can apply Lemmas\ \ref{lem:str_reg_2} and\ \ref{lem:glob_lip_H} to obtain
\begin{align}
&\left\|w^\infty\left(\bar{\mu}_k,s_{k+1}\right)-w^\ast_{k+1}\right\|_2\nonumber\\
&~~~~~~~~~~~~\leq\lambda_B\lambda_H\left\|d_\rho\left(\bar{\mu}_k\right)-d_\rho\left(\mu^\ast_{k+1}\right)\right\|_2\nonumber\\
&~~~~~~~~~~~~\leq\displaystyle\frac{\lambda_B\lambda_H}{\rho}\left(\left\|\bar{\mu}_k-\mu^\ast_k\right\|_2+\left\|\mu^\ast_{k+1}-\mu^\ast_k\right\|_2\right)\nonumber\\
&~~~~~~~~~~~~\leq\displaystyle\frac{\lambda_B\lambda_H}{\rho}\big(\big\|\bar{w}_k-w^\ast_k\big\|_2+\lambda_A\lambda_F\big\|s_{k+1}-s_k\big\|_2\big)\enspace,\nonumber
\end{align}
by Theorem \ref{theo:strg_reg_th}.
\end{IEEEproof}
In the following Lemma, using the convergence rate estimate presented in Section\ \ref{sec:theo_tools}, we derive a bound on the first summand $\left\|\bar{w}_{k+1}-w^\infty(d_\rho(\bar{\mu}_k),s_{k+1})\right\|_2$.
\begin{lem}
\label{lem:sec_term}
If $\left\|s_{k+1}-s_k\right\|_2<r_B$, $\left\|\bar{w}_k-w^\ast_k\right\|_2<q_B\rho$ and
\begin{align}
\big(1+\displaystyle\frac{\lambda_H\lambda_B}{\rho}\big)\left\|\bar{w}_k-w^\ast_k\right\|_2+\lambda_H\lambda_Br_B<\delta\enspace,\nonumber
\end{align}
where $\delta$ has been defined in Corollary\ \ref{cor:kl_al}, then  
\begin{align}
\label{eq:bnd_sec}
&\big\|\bar{w}_{k+1}-w^\infty\left(d_\rho\left(\bar{\mu}_k\right),s_{k+1}\right)\big\|_2\leq\\ 
&C\left(1+{\rho}\lambda_g\right)M^{-\psi\left(d_L,n_z\right)}\Big(\lambda_B\lambda_H\left\|s_{k+1}-s_k\right\|_2\nonumber\\
&~~~~~~~~~~~~~~~~~~~~~~~~~~~~~~~~~~~~+\big\|\bar{w}_k-w^\ast_k\big\|_2\Big(1+\displaystyle\frac{\lambda_B\lambda_H}{\rho}\Big)\Big)\enspace,\nonumber 
\end{align}
where $\lambda_G>0$ is the Lipschitz constant of $G(\cdot,s)$ on $\Zcal$ (well-defined as $\Zcal$ is bounded). 
\end{lem}
\begin{IEEEproof}
From Algorithm\ \ref{algo:splt_trck}, it follows that
\begin{align}
&\left\|\bar{w}_{k+1}-w^\infty\left(d_\rho\left(\bar{\mu}_k\right),s_{k+1}\right)\right\|_2\nonumber\\
&\leq\left\|\begin{pmatrix}
\bar{z}_{k+1}-z^\infty\left(\bar{\mu}_k,s_{k+1}\right)\\
\rho\left(G\left(\bar{z}_{k+1},s_{k+1}\right)-G\left(z^\infty\left(\bar{\mu}_k,s_{k+1}\right),s_{k+1}\right)\right)
\end{pmatrix}\right\|_2\nonumber\\
&\leq\left(1+{\rho}\lambda_G\right)\left\|\bar{z}_{k+1}-z^\infty\left(\bar{\mu}_k,s_{k+1}\right)\right\|_2\enspace.\nonumber
\end{align}
In order to apply Lemma\ \ref{lem:cv_rate}, one first need to show that $\bar{z}_k$ lies in the ball $\Bcal\big(z^\infty(\bar{\mu}_k,s_{k+1}),\delta\big)$, where $\delta$ is the radius involved in the KL property.
\begin{align}
\label{eq:ineqs_ball}
&\big\|\bar{z}_k-z^\infty(\bar{\mu}_k,s_{k+1})\big\|_2\leq\\
&~~~~~\big\|\bar{z}_k-z^\ast(0,s_k)\big\|_2+\big\|z^\ast(0,s_k)-z^\infty(\bar{\mu}_k,s_{k+1})\big\|_2\nonumber\\
&~~~~~\leq\big\|\bar{w}_k-w^\ast_k\big\|_2+\lambda_H\lambda_B\big(\left\|d_\rho(\bar{\mu}_k)\right\|_2+\left\|s_{k+1}-s_k\right\|_2\big)\nonumber\\
&~~~~~\leq\left(1+\displaystyle\frac{\lambda_H\lambda_B}{\rho}\right)\left\|\bar{w}_k-w^\ast_k\right\|_2+\lambda_H\lambda_B\left\|s_{k+1}-s_k\right\|_2\nonumber\\
&~~~~~<\delta\enspace,\nonumber
\end{align}
where the second step follows from strong regularity of\ \eqref{eq:ge_2} at $w^\ast(0,s_k)$ and the hypotheses mentioned above. Thus one can use the R-convergence rate estimate of Lemma \ref{lem:cv_rate} and apply the inequalities in\ \eqref{eq:ineqs_ball} to obtain\ \eqref{eq:bnd_sec}. 
\end{IEEEproof}
Gathering the results of Lemmas\ \ref{lem:fir_term} and\ \ref{lem:sec_term}, one can state the following theorem, which upper-bounds the sub-optimality error at time $k+1$ by a linear combination of the sub-optimality error at time $k$ and the magnitude of the parameter difference.
\begin{theo}[\label{th:wk_con}Contraction]
Given a time instant $k$, if the primal-dual error $\left\|\bar{w}_k-w^\ast_k\right\|_2$, the number of primal iterations $M$, the penalty parameter $\rho$ and the parameter difference $\left\|s_{k+1}-s_k\right\|_2$ satisfy
\begin{itemize}
\item $\left\|s_{k+1}-s_k\right\|_2<\min\left\{r_A,r_B,\displaystyle\frac{q_B\rho}{\lambda_A\lambda_F}\right\}\enspace,$
\item $\left\|\bar{w}_k-w^\ast_k\right\|_2<q_B\rho\enspace,$
\item $\rho>\tilde{\rho}\enspace,$
\item\begin{align}
\label{eq:hyp_kl_rad}
\hspace{-0.4cm}\left(1+\displaystyle\frac{\lambda_H\lambda_B}{\rho}\right)\left\|\bar{w}_k-w^\ast_k\right\|_2+\lambda_H\lambda_B\left\|s_{k+1}-s_k\right\|_2<\delta\enspace,
\end{align}  
\end{itemize}
then the following weak contraction is satisfied for all time instants $k\geq 0$:
\begin{align}
\label{eq:wk_con}
\left\|\bar{w}_{k+1}-w^\ast_{k+1}\right\|_2\leq&~\beta_w\left(\rho,M\right)\left\|\bar{w}_k-w^\ast_k\right\|_2\\
&+\beta_s\left(\rho,M\right)\left\|s_{k+1}-s_k\right\|_2\enspace,\nonumber
\end{align} where
\begin{align}
\label{eq:beta_w}
\beta_w\left(\rho,M\right):=&~C\left(1+\rho\lambda_G\right)\left(1+\displaystyle\frac{\lambda_B\lambda_H}{\rho}\right)M^{-\psi\left(d_L,n_z\right)}\nonumber\\
&+\displaystyle\frac{\lambda_B\lambda_H}{\rho}\enspace,
\end{align} and
\begin{align}
\label{eq:beta_s}
\beta_s\left(\rho,M\right):=&~C\left(1+\rho\lambda_G\right)\lambda_B\lambda_HM^{-\psi\left(d_L,n_z\right)}+\displaystyle\frac{\lambda_B\lambda_H\lambda_A\lambda_F}{\rho}\enspace.
\end{align}
\end{theo}
\begin{IEEEproof}
This is a direct consequence of Lemmas\ \ref{lem:fir_term} and\ \ref{lem:sec_term}.
\end{IEEEproof}
\begin{rk}
Note that the last hypothesis\ \eqref{eq:hyp_kl_rad} may be quite restrictive, since $\left\|\bar{w}_k-w^\ast_k\right\|_2$ needs to be small enough for it to be satisfied. However, in many cases the radius $\delta$ is large ($+\infty$ for strongly convex functions).
\end{rk}

\looseness-1In order to ensure stability of the sequence of sub-optimal iterates $\bar{w}_k$, the parameter difference $\left\|s_{k+1}-s_k\right\|_2$ has to be small enough and the coefficient $\beta_w\left(\rho,M\right)$ needs  to be strictly less than $1$.~This is clearly satisfied if the penalty parameter $\rho$ is large enough to make $\displaystyle\nicefrac{\lambda_B\lambda_H}{\rho}$ small in \eqref{eq:beta_w}.~Yet the penalty parameter $\rho$ also appears in $1+\rho\lambda_G$.~Hence it needs to be balanced by a large enough number of primal iterations $M$ in order to make the first summand in \eqref{eq:beta_w} small.~The same analysis applies to the second coefficient $\beta_s\left(\rho,M\right)$ in order to mitigate the effect of the parameter difference $\left\|s_{k+1}-s_k\right\|_2$ on the sub-optimality error at $k+1$.
\begin{cor}[\label{cor:bnd_err}Boundedness of the error sequence]
Assume that $\rho$ and $M$ have been chosen so that $\beta_w\left(\rho,M\right)$ and $\beta_s\left(\rho,M\right)$ are strictly less than $1$, and $\rho>\tilde{\rho}$. Let $r_w>0$ such that 
\begin{align}
\delta-\left(1+\displaystyle\frac{\lambda_H\lambda_B}{\rho}\right)r_w>0
\end{align}
and $r_w<q_B\rho$. Let $r_s>0$ such that 
\begin{align}
r_s<\displaystyle\frac{(1-\beta_w(\rho,M))r_w}{\beta_s(\rho,M)}\enspace.
\end{align}
If $\left\|\bar{w}_0-w^\ast_0\right\|_2<r_w$ and for all $k\geq0$,
\begin{align}
\label{eq:hyp_prm_diff}
\left\|s_{k+1}-s_k\right\|_2\leq\min\left\{r_s,r_A,r_B,\displaystyle\frac{q_B\rho}{\lambda_A\lambda_F}\right\}\enspace,
\end{align}
then for all $k\geq0$, the error sequence satisfies 
\begin{align}
\left\|\bar{w}_k-w^\ast_k\right\|_2<r_w\enspace.
\end{align}
\end{cor}
\begin{IEEEproof}
The proof proceeds by a straightforward induction.~At $k=0$, $\left\|\bar{w}_0-w^\ast_0\right\|_2<r_w$, by assumption.~Let $k\geq0$ and assume that $\left\|\bar{w}_k-w^\ast_k\right\|_2<r_w$.~As $\left\|s_{k+1}-s_k\right\|_2<r_A$, by applying Theorem\ \ref{theo:strg_reg_th}, there exists a unique $w^\ast_{k+1}\in\Bcal\left(w^\ast_k,\delta_A\right)$, which satisfies\ \eqref{eq:ge_1}.~As $\left\|s_{k+1}-s_k\right\|_2$ satisfies \eqref{eq:hyp_prm_diff}, $\left\|\bar{w}_k-w^\ast_k\right\|_2<q_B\rho$, $\rho>\tilde{\rho}$ and \eqref{eq:hyp_kl_rad} is satisfied, from the choice of $r_w$ and $r_s$, we have
\begin{align}
\left\|\bar{w}_{k+1}-w^\ast_{k+1}\right\|_2&\leq\beta_w\left(\rho,M\right)\big\|\bar{w}_k-w^\ast_k\big\|_2\\
&~~~~+\beta_s\left(\rho,M\right)\left\|s_{k+1}-s_k\right\|_2\nonumber\\
&\leq\beta_w\left(\rho,M\right)r_w+\beta_s\left(\rho,M\right)\left\|s_{k+1}-s_k\right\|_2\nonumber\\
&\leq r_w\enspace,
\end{align}
as $\left\|s_{k+1}-s_k\right\|_2\leq r_s<\displaystyle\nicefrac{(1-\beta_w(\rho,M))r_w}{\beta_s(\rho,M)}$.~Note from the choice of $r_w$ and $r_s$, the condition\ \eqref{eq:hyp_kl_rad} guaranteeing the weak contraction\ \eqref{eq:wk_con} is also recursively satisfied.
\end{IEEEproof}

\subsection{Improved contraction via continuation}
\label{subsec:hom_con}

\looseness-1In Algorithm\ \ref{algo:splt_trck}, only one dual update is performed to move from parameter $s_k$ to parameter $s_{k+1}$, in contrast to standard augmented Lagrangian techniques where the Lagrange multiplier $\mu$ and the penalty parameter $\rho$ are updated after every sequence of primal iterations.~Intuitively, one would expect that applying several dual updates instead of just one, drives the suboptimal solution $\bar{w}_{k+1}$ closer to the optimal one $w^\ast_{k+1}$, thus enhancing the tracking performance over time.~However, as the number of primal iterations $M$ is fixed a priori, it is not obvious at all why this would happen, as primal iterations generally need to become more accurate when the dual variable moves closer to optimality.~Therefore, we resort to an homotopy-based mechanism to fully exploit property\ \eqref{eq:wk_con}.~Continuation techniques\ \cite{geo1990}, in which the optimal solution is parameterised along an homotopy path from the previous parameter $s_k$ to the current one $s_{k+1}$, have been successfully applied for solving online convex quadratic programs in the\ \textsc{qpoases} package\ \cite{fer2008}. 

\looseness-1The parameter $s$ can be seen as an extra degree of freedom in Algorithm\ \ref{algo:splt_trck}, which can be modified along the iterations.~More precisely, instead of carrying out a sequence of alternating proximal gradient steps to find a critical point of $L_\rho\left(\cdot,\bar{\mu}_k,s_{k+1}\right)+\iota_\Zcal$ directly at the parameter $s_{k+1}$, one moves from $s_k$ towards $s_{k+1}$ step by step, each step corresponding to a dual update and a sequence of alternating proximal gradients.~The proposed approach can be seen as a form of `tracking in the tracking'.~More precisely, one defines a finite sequence $\left\{s_k^j\right\}$ of $D$ parameter along the homotopy path $\left\{\left(1-\tau\right)s_k+{\tau}s_{k+1}~\big|~\tau\in\left[0,1\right]\right\}$ by       
\begin{align}
\label{eq:s_j_k}
s^j_k:=\left(1-\displaystyle\frac{j}{D}\right)s_k+\displaystyle\frac{j}{D}s_{k+1},~j\in\left\{0,\ldots,D\right\}\enspace,
\end{align}
where $D\geq2$.~This modification results in Algorithm\ \ref{algo:hom_splt_trck} below.~At every step $j$, an homotopy update is first carried out. A sequence of proximal minimisation is then applied given the current parameter $s$ and multiplier $\mu$, which is updated at the end of step $j$.~In a sense, Algorithm\ \ref{algo:hom_splt_trck} consists in repeatedly applying Algorithm\ \ref{algo:splt_trck} on an artificial dynamics determined by the homotopy steps. 
\begin{algorithm}[h!]
	\caption{\label{algo:hom_splt_trck}Homotopy-based optimality tracking splitting algorithm}
	\begin{algorithmic}
		\State \textbf{Input:} Suboptimal primal-dual solution $\left(\bar{z}\left(s_k\right)^{\Trans},\bar{\mu}_k^{\Trans}\right)^{\Trans}$, parameters $s_k$ and $s_{k+1}$.
		\State $s \leftarrow s_k$, $\mu\leftarrow\bar{\mu}_k$, $z^{\texttt{wms}}\leftarrow\bar{z}_k$
		\State \underline{\textbf{Continuation loop}}:
		\For{$j=1\ldots D$}
			\State $s \leftarrow s+\displaystyle\frac{s_{k+1}-s_k}{D}$
			\State \underline{\textbf{Primal loop}}:	
			\State $z^{(0)}\leftarrow z^{\texttt{wms}}$
			\For{$l=0\ldots M-1$}
				\For{i=1\ldots P}
					\State $z_i^{(l+1)}\leftarrow\texttt{bckMin}\left(z_i^{(l)},\mu,\rho,s\right)$
				\EndFor
			\EndFor
			\State $z^{\texttt{wms}} \leftarrow z^{(M)}$
			\State \underline{\textbf{Dual update}}: $\mu\leftarrow\mu+\rho G\left(z^{(M)},s\right)$
		\EndFor
		\State $\bar{z}\left(s_{k+1}\right)\leftarrow z^{\texttt{wms}}$; $\bar{\mu}_{k+1}\leftarrow\mu$
		\end{algorithmic}
\end{algorithm}

The reason for introducing Algorithm\ \ref{algo:hom_splt_trck} is that it allows for a stronger contraction effect on the sub-optimality gap $\left\|\bar{w}_{k+1}-w^{\ast}_{k+1}\right\|_2$ than Algorithm\ \ref{algo:splt_trck}, as formalised by the following Theorem. 
\begin{lem}[\label{lem:hom_opt}Optimality along the homotopy path]
Given a time instant $k\geq0$, for all $j\in\left\{1,\ldots,D\right\}$, there exists a unique primal-dual variable $w^\ast\left(s^j_k\right)\in\Bcal\left(w_k^\ast,r_A\right)$ satisfying
\begin{align}
0\in F\left(w^\ast\left(s^j_k\right),s^j_k\right)+\Ncal_{\Zcal\times\Rset^m}\left(w^\ast\left(s^j_k\right)\right)\enspace.
\end{align} 
\end{lem}
\begin{IEEEproof}
This comes directly from the strong regularity of\ \eqref{eq:ge_1}, Assumption\ \ref{ass:prm_diff} and $\left\|s_k^j-s_k\right\|_2\leq\left\|s_{k+1}-s_k\right\|_2$ for all $j\in\left\{1,\ldots,D\right\}$.
\end{IEEEproof}
\begin{rk}
Note that the NMPC program\ \eqref{eq:prm_nlp_1} at parameter $s^j_k$, $j\in\left\{1,\ldots,D\right\}$, is feasible, by strong regularity of\ \eqref{eq:ge_1} at $w^\ast\left(s^0_k\right)$, since $\left\|s^j_k-s_k\right\|_2<r_A$. However, in general, for an arbitrarily large parameter difference $\left\|s_{k+1}-s_k\right\|_2$, this is not true, as the feasible set of the NMPC controller associated with\ \eqref{eq:prm_nlp_1} is not convex.  
\end{rk}
\begin{theo}[Improved contraction via continuation]
Assume that $\rho>\tilde{\rho}$ and that $\rho$ and $M$ have been chosen so that $\beta_w\left(\rho,M\right),\beta_s\left(\rho,M\right)<1$.~Given a time instant $k\geq0$, if $\left\|\bar{w}_k-w^\ast_k\right\|_2<r_w$, where $r_w$ satisfies the assumptions of Corollary\ \ref{cor:bnd_err}, and $\left\|s_{k+1}-s_k\right\|_2$ satisfies\ \eqref{eq:hyp_prm_diff}, then the primal-dual sub-optimal variable $\bar{w}_{k+1}$ yielded by Algorithm \ref{algo:hom_splt_trck} satisfies the following inequality 
\begin{align}
\label{eq:imp_ineq}
\left\|\bar{w}_{k+1}-w^{\ast}_{k+1}\right\|_2&\leq\beta_w^D\left(\rho,M\right)\left\|\bar{w}_k-w^{\ast}_k\right\|_2\\
&+\beta_s\left(\rho,M\right)\displaystyle\frac{\sum_{i=0}^{D-1}\beta_w^i\left(\rho,M\right)}{D}\left\|s_{k+1}-s_k\right\|_2\enspace\nonumber
\end{align}
\end{theo}
\begin{IEEEproof}
For all $j\in\left\{1,\ldots,D\right\}$, define 
\begin{align}
\bar{\mu}^j_k:=\bar{\mu}^{j-1}_k+\rho{}G\left(\bar{z}_k^j,s_k^j\right)\nonumber
\end{align} 
with $\bar{\mu}^0_k:=\bar{\mu}_k$ and where $\bar{z}_k^j$ is obtained after $M$ alternating proximal gradient steps applied to $L_\rho\left(\cdot,\bar{\mu}^{j-1}_k,s_k^j\right)+\iota_\Zcal$.~One can thus define a sub-optimal primal-dual variable 
\begin{align}
\nonumber
\bar{w}_k^j:=\left(\left(\bar{z}_k^j\right)^{\Trans},\left(\bar{\mu}_k^j\right)^{\Trans}\right)^{\Trans}
\end{align}
for the homotopy parameter $s^j_k$.~By applying Corollary\ \ref{cor:bnd_err}, one obtains that for all $j\in\left\{0,\ldots,D-1\right\}$, 
\begin{align}
\nonumber
\left\|\bar{w}_k^j-w^\ast\left(s^j_k\right)\right\|_2< r_w<q_B\rho\enspace,
\end{align} 
since 
\begin{align}
\nonumber
\left\|s^{j+1}_k-s^j_k\right\|_2=\displaystyle\frac{\left\|s_{k+1}-s_k\right\|_2}{D}<\min\left\{r_s,r_A,r_B,\frac{q_B\rho}{\lambda_A\lambda_F}\right\}\enspace.
\end{align}
It can also be readily shown that for all $j\in\left\{0,\ldots,D-1\right\}$,
\begin{align} 
\left(1+\displaystyle\frac{\lambda_H\lambda_B}{\rho}\right)\left\|\bar{w}_k^j-w^\ast\left(s^j_k\right)\right\|_2+\lambda_H\lambda_B\left\|s^{j+1}_k-s^j_k\right\|_2<\delta\enspace.
\end{align}
Subsequently, one can apply the same reasoning as for proving Theorem\ \ref{th:wk_con}, and get that for all $j\in\left\{0,\ldots,D-1\right\}$, 
\begin{align}
\label{eq:hom_con_j}
&\left\|\bar{w}_k^{j+1}-w^\ast\left(s_k^{j+1}\right)\right\|_2\leq\\
&~~~~~\beta_w\left(\rho,M\right)\left\|\bar{w}_k^j-w^\ast\left(s_k^j\right)\right\|_2+\beta_s\left(\rho,M\right)\left\|s^{j+1}_k-s^j_k\right\|_2\enspace.\nonumber
\end{align}
By iterating inequality\ \eqref{eq:hom_con_j} from $j=0$ to $D-1$, we obtain
\begin{align}
 \left\|\bar{w}_{k+1}-w^\ast_{k+1}\right\|_2&\leq\beta_w\left(\rho,M\right)\left\|\bar{w}_k^{D-1}-w^\ast\left(s_k^{D-1}\right)\right\|_2\nonumber\\
 &~~~~+\displaystyle\frac{\beta_s\left(\rho,M\right)}{D}\left\|s_{k+1}-s_k\right\|_2\nonumber\\
 &\leq\ldots\nonumber\\
 &\leq\beta_w^D\left(\rho,M\right)\left\|\bar{w}^0_k-w^{\ast}\left(s^0_k\right)\right\|_2\nonumber\\
 &+\beta_s\left(\rho,M\right)\displaystyle\frac{\sum_{j=0}^{D-1}\beta_w^j\left(\rho,M\right)}{D}\left\|s_{k+1}-s_k\right\|_2\enspace\nonumber
 \end{align}
 which is exactly inequality\ \eqref{eq:imp_ineq}.
\end{IEEEproof}
As $\beta_w\left(\rho,M\right)<1$, $\beta_s\left(\rho,M\right)<1$ and $D\geq2$, it follows that 
\begin{align}
\nonumber
\beta_w^D\left(\rho,M\right)<\beta_w\left(\rho,M\right)
\end{align}
and
\begin{align} 
\beta_s\left(\rho,M\right)\displaystyle\frac{\sum_{i=0}^{D-1}\beta_w^i\left(\rho,M\right)}{D}<\beta_s\left(\rho,M\right)\enspace,\nonumber
\end{align} 
which implies that the contraction\ \eqref{eq:imp_ineq} is stronger than\ \eqref{eq:wk_con}.~In practice, the coefficients $\beta_w\left(\rho,M\right)$ and $\beta_s\left(\rho,M\right)$ in\ \eqref{eq:wk_con} can be reduced by an appropriate tuning of the penalty $\rho$ and the number of primal proximal steps $M$.~Yet this approach is limited, as previously discussed in Paragraph \ref{subsec:stab_ana}.~Therefore, Algorithm\ \ref{algo:hom_splt_trck} provides a more efficient and systematic way of improving the optimality tracking performance.~Superiority of Algorithm\ \ref{algo:hom_splt_trck} over Algorithm\ \ref{algo:splt_trck} is demonstrated on a numerical example in Section\ \ref{sec:num_ex}.


\section{Computational considerations}
\label{sec:comp_asp}
\looseness-1By making use of partial penalisation, Algorithm\ \ref{algo:splt_trck} allows for a more general problem formulation than\ \cite{zav2010}, where the primal QP sub-problem is assumed to have non-negativity constraints only.~Moreover, the approach of\ \cite{zav2010} is likely to be efficient only when the sub-optimal solution is close enough to the optimal manifold so as to guarantee positive definiteness of the hessian of the augmented Lagrangian.~In practice, this cannot always be ensured if the reference change or disturbances are too large.~In contrast, our framework can handle any polynomial non-convex objective subject to convex constraint set $\Zcal_i$ for which the proximal operator\ \eqref{eq:prox} is cheap to compute.~This happens when $\Zcal_i$ is a ball, an ellipsoid, a box, the non-negative orthant or even second-order conic constraints and the semi-definite cone.~However, the theoretical properties derived in Section\ \ref{sec:con_prop} seem to be limited to polyhedral constraint sets.
\begin{rk}
\looseness-1For many non-convex sets, such as spheres or mixed-integer sets, the proximal operator\ \eqref{eq:prox} can be obtained in closed-form.~However, the analysis of Section\ \ref{sec:con_prop} does not readily extend, as Robinson's strong regularity is defined for closed convex sets\ \cite{robin1980}.
\end{rk}
\begin{rk}
In a distributed framework, any convex polyhedral set $\Zcal_i$ could be handled by Algorithm\ \ref{algo:splt_trck}, as a non-negative slack variable can be introduced for every agents.  
\end{rk}
Algorithm\ \ref{algo:splt_trck} can be further refined by introducing local copies of the variables.~Considering the NLP
\begin{align}
\minimise~&J(z_1,\ldots,z_P)\nonumber\\
\text{s.t.}~&G(z_1,\ldots,z_P)=0\nonumber\\
&z_1\in\Zcal_1,\ldots,z_P\in\Zcal_P\enspace,\nonumber
\end{align}
variables $y_i$ can be incorporated in the equality constraints, resulting in
\begin{align}
\minimise~&J(y_1,\ldots,y_P)\nonumber\\
\text{s.t.}~&G(y_1,\ldots,y_P)=0\nonumber\\
&y_i-z_i=0~~\forall i\in\left\{1,\ldots,P\right\}\nonumber\\
&z_1\in\Zcal_1,\ldots,z_P\in\Zcal_P\enspace.\nonumber
\end{align}
Subsequently, at iteration $l+1$, some of the alternations are given by 
\begin{align}
\minimise_{z_i\in\Zcal_i}~\nu_i^{\Trans}\left(y_i^{(l+1)}-z_i\right)&+\frac{\rho}{2}\left\|y_i^{(l+1)}-z_i\right\|_2^2\nonumber\\
&~~~~~~~~~~~~+\frac{\alpha_i}{2}\left\|z_i-z_i^{(l)}\right\|_2^2\enspace,\nonumber
\end{align}
where $\nu_i$ is a dual variable associated with the equality constraint $y_i-z_i=0$.~This step can be rewritten 
\begin{align} 
\minimise_{z_i\in\Zcal_i}~\left\|z_i-\frac{1}{\alpha_i+\rho}\left(\alpha_iz_i^{(l)}+{\rho}y_i^{(l+1)}+\nu_i\right)\right\|_2\enspace,\nonumber
\end{align}
which corresponds to projecting 
\begin{align}
\frac{1}{\alpha_i+\rho}\left(\alpha_iz_i^{(l)}+{\rho}y_i^{(l+1)}+\nu_i\right)\nonumber
\end{align}
onto $\Zcal_i$.~This type of an approach is useful if the minimisation over the $y_i$ variables is tractable, for instance when $J$ is multi-convex and $G$ is multilinear, and the projection onto $\Zcal_i$ is cheap to compute.
\section{Numerical examples}
\label{sec:num_ex}
\looseness-1Algorithms\ \ref{algo:splt_trck} and\ \ref{algo:hom_splt_trck} are tested on two nonlinear systems, a DC motor (centralised) in paragraph\ \ref{subsec:dc_mot} and a formation of three unicycles (distributed) in paragraph\ \ref{subsec:coll_track}.~The effect of the penalty parameter $\rho$ and the sampling period ${\Delta}t$ is analysed, assuming that a fixed number of iterations can be performed per second.~Thus, given a sampling period ${\Delta}t$, the number of communications between the $P$ groups of agents is limited to a fixed value, which models practical limitations of distributed computations.~In particular, it is shown that the theoretical results proven in Section \ref{sec:con_prop} are able to predict the practical behaviour of the combined system-optimiser dynamics quite well, and that tuning the optimiser's step-size $\rho$ and the system's step-size ${\Delta}t$ has an effect on the closed-loop trajectories.

\looseness-1From a more practical perspective, the purpose of the simulations that follow is to investigate the effect of a limited computational power and limited communication rate on the closed-loop performance of our scheme.~This is of particular importance in the case of distributed NMPC problems, as in practice, only a limited number of packets can be exchanged between the $P$ groups of agents within a fixed amount of time, which implies that a suboptimal solution is yielded by Algorithm\ \ref{algo:splt_trck}.

\begin{rk} 
In the following examples, the first optimal primal-dual solution $w^{\ast}_0$ is computed using the distributed algorithm proposed in\ \cite{hou2014}.~A random perturbation is then applied to this KKT point.
\end{rk}

\subsection{DC motor}
 \label{subsec:dc_mot}
The first example is a DC motor with continuous-time bilinear dynamics
\begin{align}
\dot{x}=Ax+Bx\cdot u+c\enspace,\nonumber
\end{align}
where 
\begin{align}
&A=\begin{pmatrix}
-\nicefrac{R_a}{L_a} & 0 \\
0 & -\nicefrac{B}{J}
\end{pmatrix}\enspace, B=\begin{pmatrix}
0 & -\nicefrac{k_m}{L_a} \\
\nicefrac{k_m}{J} & 0
\end{pmatrix}\enspace,\nonumber\\
&c=\begin{pmatrix}
\nicefrac{u_a}{L_a}\\
\nicefrac{-\tau_l}{J}
\end{pmatrix}\enspace,\nonumber
\end{align}
and the parameters are borrowed from the experimental identification presented in\ \cite{dan1998}:
\begin{align}
L_a&=0.307~\text{H},~R_a=12.548~\Omega,~k_m=0.22567~\nicefrac[]{\text{Nm}}{\text{A}^2}\enspace,\nonumber\\
J&=0.00385~\text{Nm.sec}^2,~B=0.00783~\text{Nm.sec}\enspace,\nonumber\\
\tau_l&=1.47~\text{Nm},~u_a=60~\text{V}\enspace.\nonumber
\end{align}
\looseness-1The first component of the state variable $x_1$ is the armature current, while the second component $x_2$ is the angular speed.~The control input $u$ is the field current of the machine.~The control objective is to make the angular speed track a piecewise constant reference $x_2^{\text{ref}}=\pm2~\nicefrac[\text]{rad}{sec}$, while satisfying the following state and input constraints:
\begin{align}
&\underline{x}=\begin{pmatrix}
-2~\text{A}\\ -8~\text{rad/sec}
\end{pmatrix},~\overline{x}=\begin{pmatrix}
5~\text{A}\\ 1.5~\text{rad/sec}
\end{pmatrix}\enspace,\nonumber\\
&\underline{u}=1.27~\text{A},~\overline{u}=1.4~\text{A}\enspace.\nonumber
\end{align}
\looseness-1The continuous-time NMPC problem for reference tracking is discretised at a given sampling period ${\Delta}t$ using an explicit Euler method, which results in a bilinear NLP.~Although the consistency of the explicit Euler integrator is $1$, only the first control input is applied to the real system, implying that the prediction error with respect to the continuous-time dynamics is small.~For simulating the closed-loop system under the computed NMPC control law, the\ \textsc{matlab} integrator\ \texttt{ode45} is used with the sampling period ${\Delta}t$.~The prediction horizon is fixed at $30$ samples.~This is a key requirement for the analysis that follows, as explained later.
\begin{figure}[h!]
\begin{center}
\includegraphics[width=1\columnwidth]{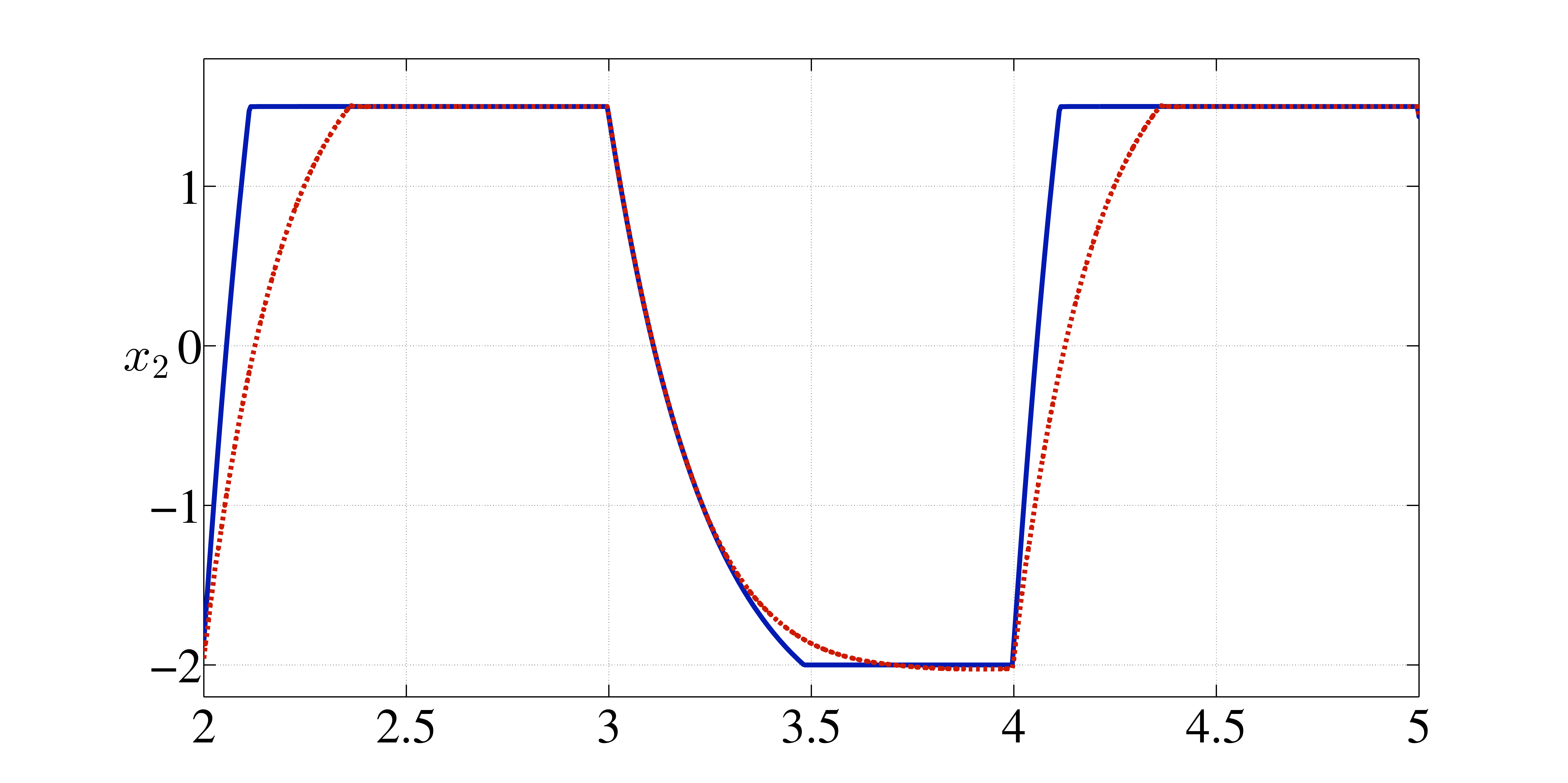}
\end{center}
\vspace{-0.4cm}
\begin{center}
\includegraphics[width=1\columnwidth]{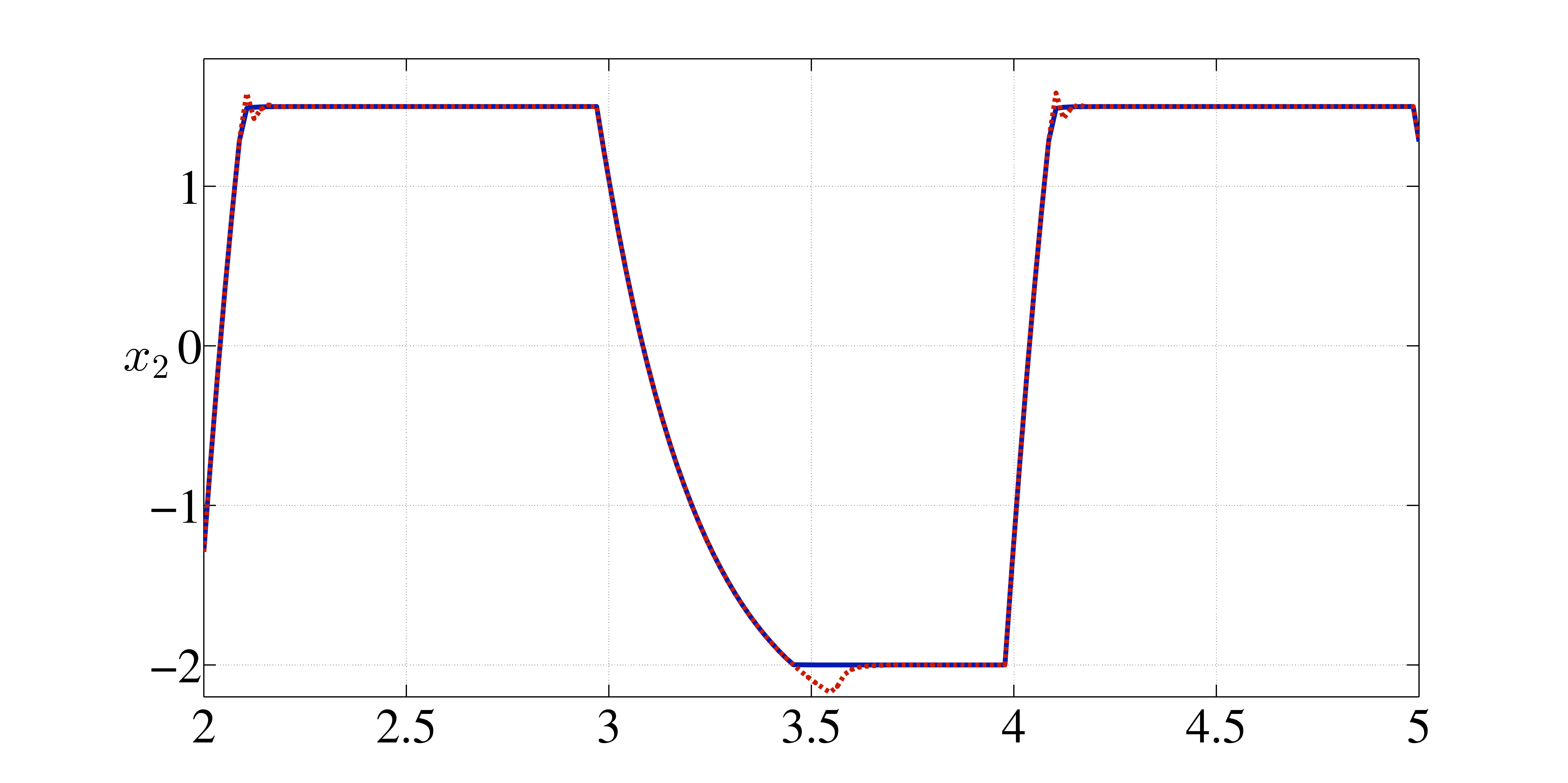}
\end{center}
\vspace{-0.4cm}
\begin{center}
\includegraphics[width=1\columnwidth]{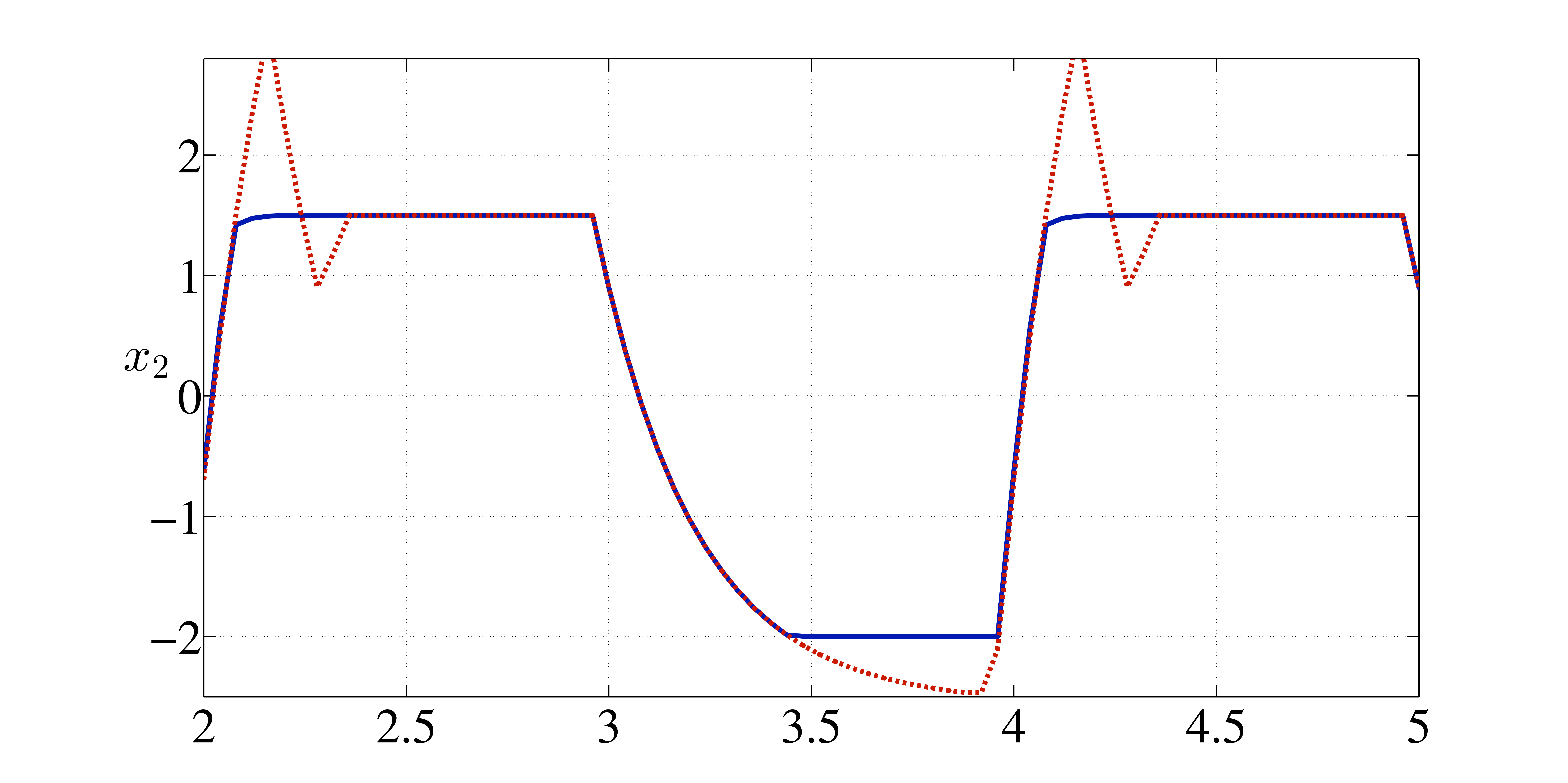}
\vspace{-0.4cm}
\put(0,0){\footnotesize Time (s)}
\end{center}
\caption{\label{fig:angSpeed_dt}Angular speed against time for increasing sampling periods ${\Delta}t$ and a fixed computational power $2\cdot10^3~\nicefrac{\text{prox}}{\text{sec}}$: $0.004$ sec (top), $0.018$ sec (middle) and $0.04$ sec (bottom).~The sub-optimal trajectory obtained with Algorithm\ \ref{algo:splt_trck} is plotted in dashed red, while the full NMPC trajectory obtained using\ \textsc{ipopt} (for the same ${\Delta}t$) is in blue.}
\end{figure}
\begin{figure}[h!]
\begin{center}
\includegraphics[width=1\columnwidth]{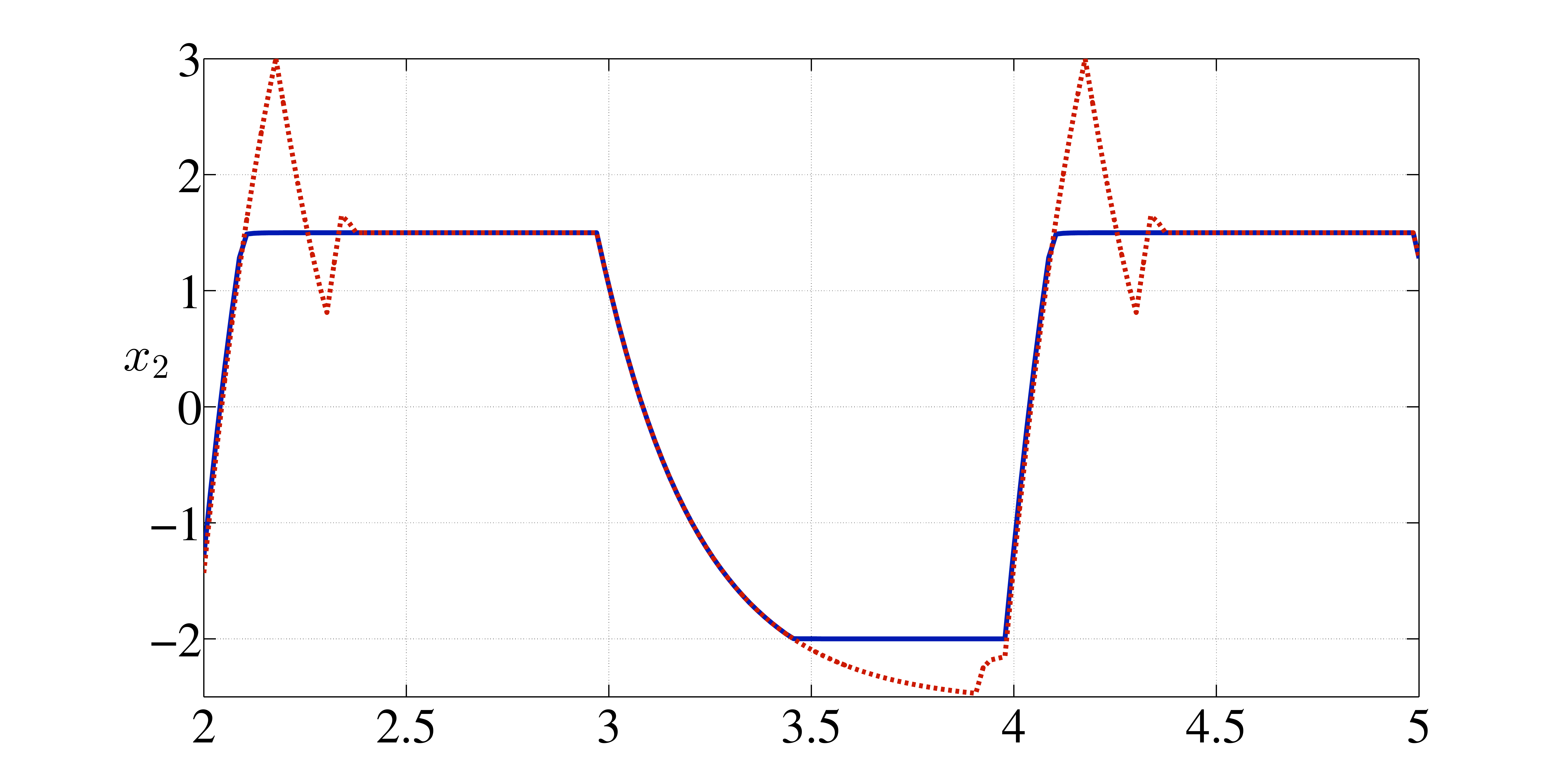}
\vspace{-0.4cm}
\includegraphics[width=1\columnwidth]{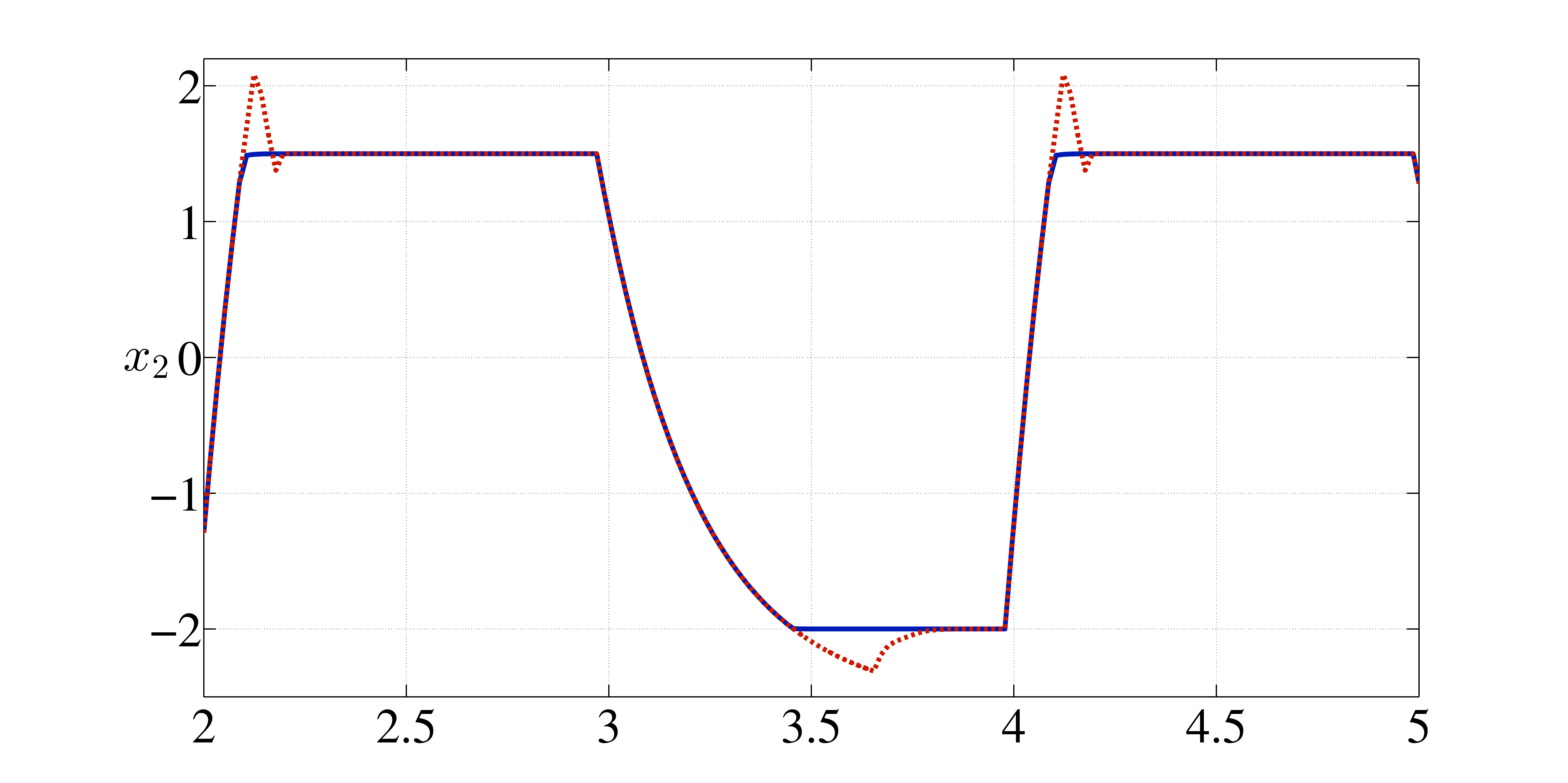}
\vspace{-0.4cm}
\includegraphics[width=1\columnwidth]{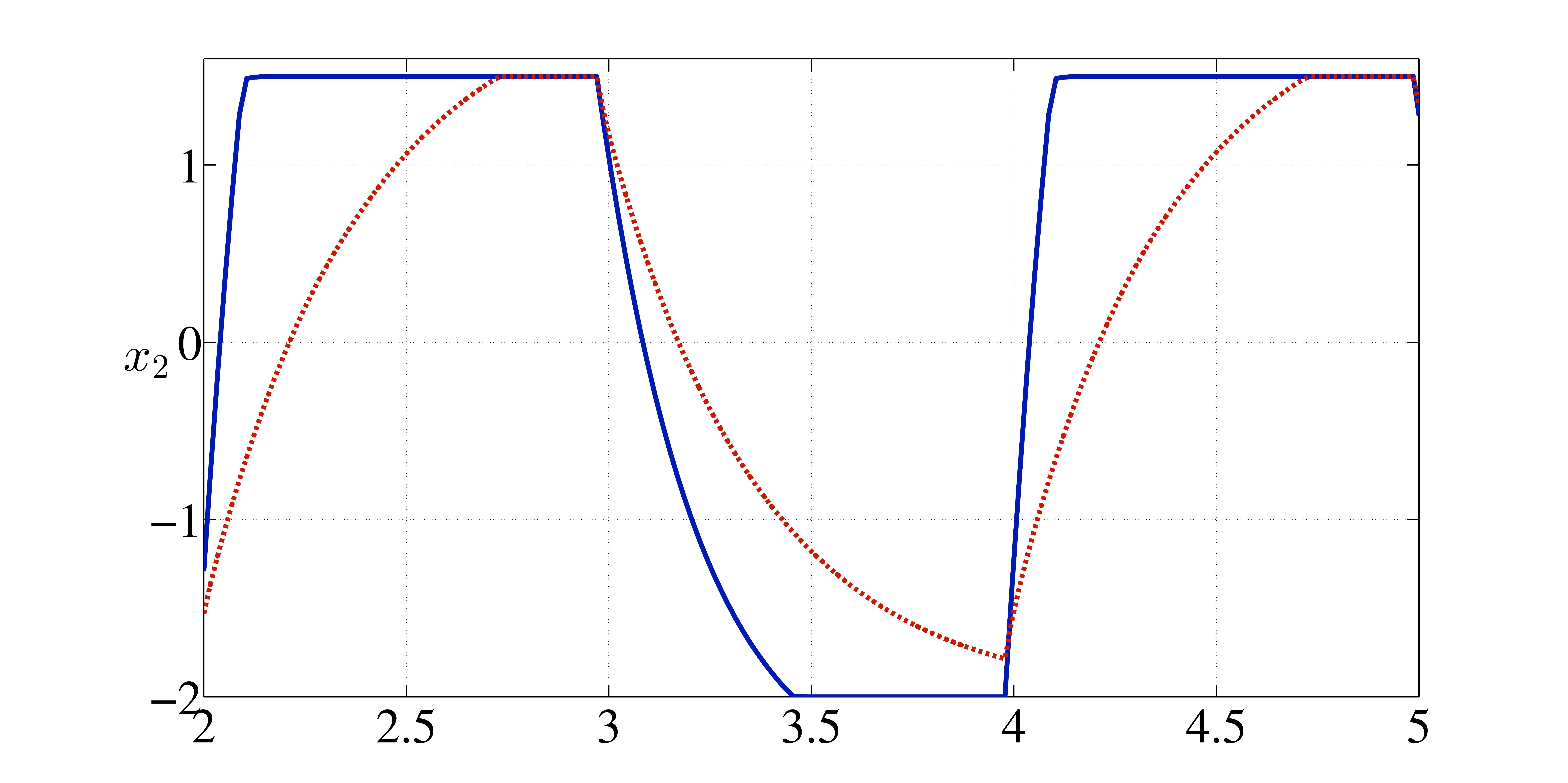}
\vspace{-0.4cm}
\put(0,0){\footnotesize Time (s)}
\end{center}
\caption{\label{fig:angSpeed_rho}Angular speed against time for increasing penalty parameters $\rho$ and a fixed computation power $2\cdot10^3~\nicefrac{\text{prox}}{\text{sec}}$: $20$ (top), $100$ (middle) and $1000$ (bottom).~The sub-optimal trajectory obtained with Algorithm\ \ref{algo:splt_trck} is plotted in dashed red, while the full NMPC trajectory obtained using\ \textsc{ipopt} (for the same ${\Delta}t$) is in blue.}
\end{figure}

\looseness-1In general, the computational power of an embedded computing platform is quite limited, meaning that the total number of proximal steps that can be computed within one second by Algorithms\ \ref{algo:splt_trck} and\ \ref{algo:hom_splt_trck} is fixed and finite.~Later on, we refer to this number as the\ \textit{computational power}, expressed in $\nicefrac{\text{prox}}{\text{sec}}$.~The results plotted in Figs.\ \ref{fig:angSpeed_dt}, \ref{fig:angSpeed_rho},\ \ref{fig:kkt} and\ \ref{fig:feas} are obtained for a computational power of $2\cdot10^3~\nicefrac{\text{prox}}{\text{sec}}$.~In Fig.\ \ref{fig:angSpeed_dt}, it clearly appears that a better tracking performance is obtained for ${\Delta}t=0.018$ sec, compared to a lower sampling period (${\Delta}t=0.004$ sec) or a larger sampling period (${\Delta}t=0.04$ sec).~The effect of the system's step-size ${\Delta}t$ on the performance of Algorithm\ \ref{algo:splt_trck} given a fixed computational power is demonstrated more clearly in Fig.\ \ref{fig:par_dc_palm}.

\looseness-1Another key parameter is the penalty coefficient $\rho$, which can also be interpreted as a step-size for the optimiser.~In order to demonstrate the effect of $\rho$ on the efficacy of our optimality tracking splitting scheme, the sampling period ${\Delta}t$ is fixed at $0.018$ sec given a computational power of $2\cdot10^3~\nicefrac{\text{prox}}{\text{sec}}$, which implies that the total number of primal iterations is $M=36$, and $\rho$ is made vary within $\left\{20,100,1\cdot10^3\right\}$.~Figure\ \ref{fig:angSpeed_rho} shows that a better tracking performance is obtained with $\rho=100$ than with $\rho=20$ or $\rho=1\cdot10^3$.~This can be deduced from the expression of the coefficient $\beta_w\left(\rho,M\right)$ in Eq.\ \eqref{eq:beta_w}, as explained in Paragraph\ \ref{subsec:stab_ana}.~The optimal choice of the penalty parameter is known to be critical to the convergence speed of ADMM, which is very similar to our optimality tracking splitting scheme, since it can be interpreted as a truncated Gauss-Seidel procedure in an augmented Lagrangian.~To our knowledge, this effect has only been observed for ADMM-type techniques when dealing with convex programs.~When solving non-convex programs using augmented Lagrangian techniques, it is commonly admitted that $\rho$ should be chosen large enough in order to ensure (locally) positive definiteness of the hessian of the augmented Lagrangian.~Taking $\rho$ too large is known to result in ill-conditioning.~For Algorithm\ \ref{algo:splt_trck}, which is essentially a first-order method, the analysis is different, as $\rho$ does not affect the algorithm at the level of linear algebra, but does impact the contraction of the primal-dual sequence, and thus the convergence speed over time, or tracking performance.~Thus our study provides a novel interpretation of the choice of $\rho$ via a parametric analysis in a non-convex framework.~The effect of the optimiser step-size $\rho$ on the closed-loop performance fully appears in Fig.\ \ref{fig:rho_eff}.   
\begin{figure}[h!]
\begin{center}
\includegraphics[width=1\columnwidth]{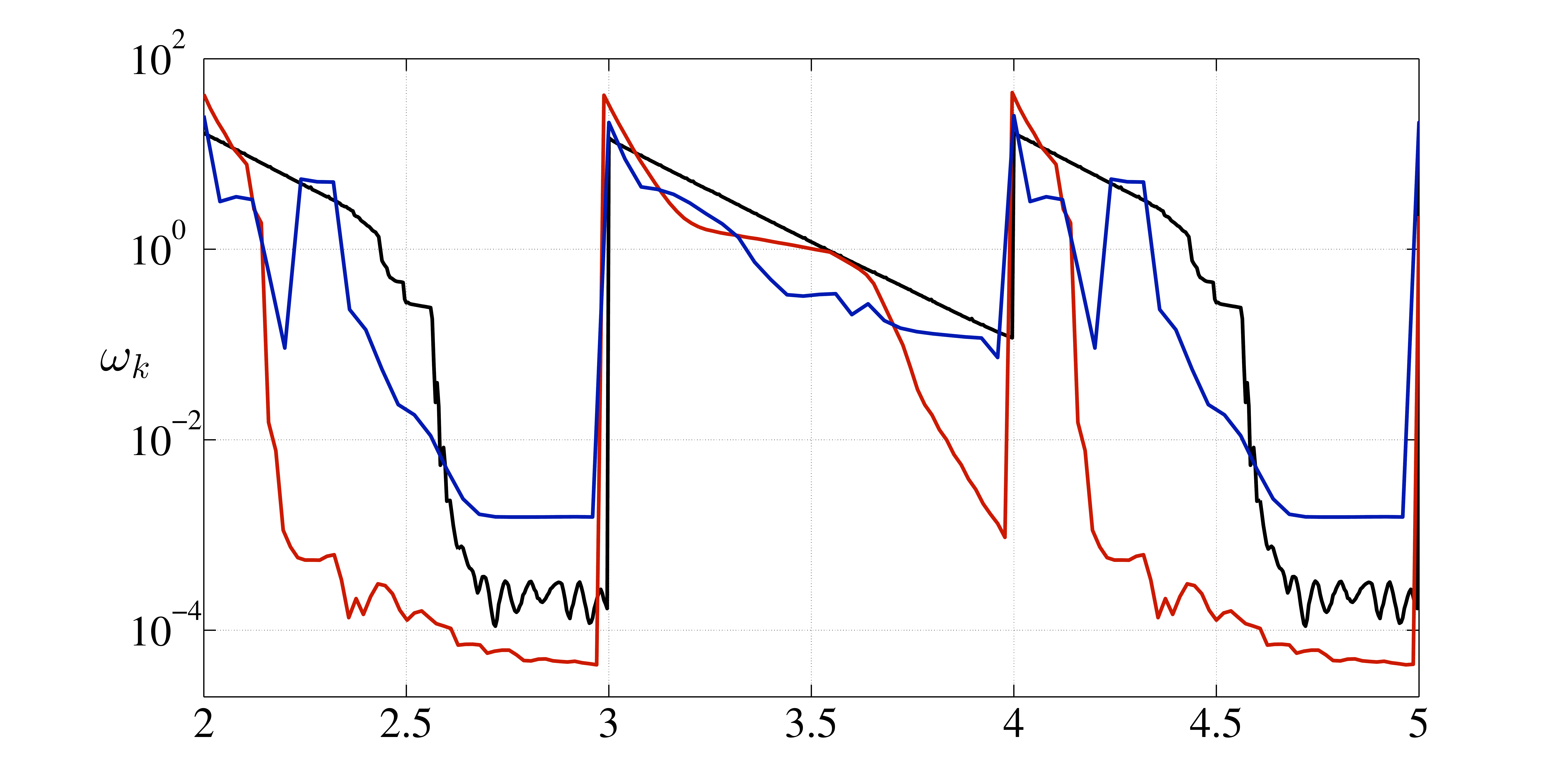}
\vspace{-0.4cm}
\put(0,5){\footnotesize Time (s)}
\end{center}
\caption{\label{fig:kkt}Optimality of bound constrained augmented Lagrangian program for different sampling periods ${\Delta}t$ and a fixed computation power $2\cdot10^3~\nicefrac{\text{prox}}{\text{sec}}$: $0.004$ sec (black), $0.018$ sec (red) and $0.04$ sec (blue).}
\end{figure}
Satisfaction of the KKT conditions of the parametric augmented Lagrangian problem 
\begin{align}
\minimise_{z\in B\left(\underline{z},\overline{z}\right)}~L_\rho\left(z,\bar{\mu}_k,s_{k+1}\right)\nonumber
\end{align}
is measured along the closed-loop trajectory by computing 
\begin{align}
\omega_k:=\|\pi_{B\left(\underline{z},\overline{z}\right)}\left(\bar{z}\left(\bar{\mu}_{k-1},s_k\right)-{\nabla}L_\rho\left(\bar{z}\left(\bar{\mu}_{k-1},s_k\right),\bar{\mu}_{k-1},s_k\right)\right)\nonumber\\
-\bar{z}\left(\bar{\mu}_{k-1},s_k\right)\|_2,\enspace\nonumber
\end{align}
\looseness-1which is plotted in Fig.\ \ref{fig:kkt}.~Over time, convergence towards low criticality values is faster for ${\Delta}t=0.18$ sec, than for shorter sampling period (${\Delta}t=0.004$ sec) or larger sampling period (${\Delta}t=0.04$ sec).~The same effect can be observed for the feasibility of the nonlinear equality constraints $G\left(\cdot,s_k\right)$, as pictured in Fig.\ \ref{fig:feas}. 
\begin{figure}[h!]
\begin{center}
\includegraphics[width=1.1\columnwidth]{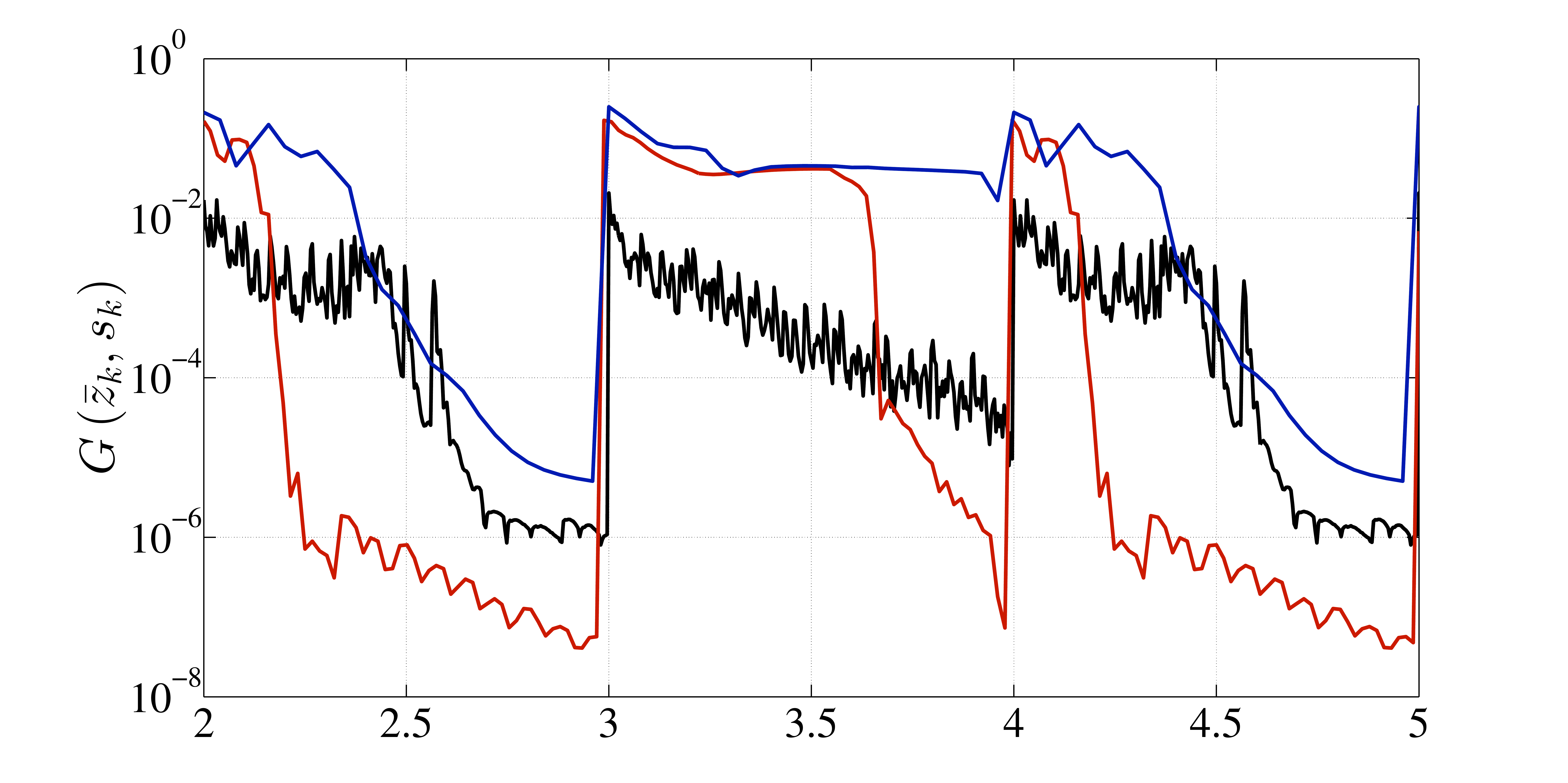}
\vspace{-0.4cm}
\put(0,5){\footnotesize Time (s)}
\end{center}
\caption{\label{fig:feas}Norm of equality constraints $\left\|G\left(\bar{z}_k,s_k\right)\right\|_2$ for different sampling periods ${\Delta}t$ and a fixed computation power $2\cdot10^3~\nicefrac{\text{prox}}{\text{sec}}$: $0.004$ sec (black), $0.018$ sec (red) and $0.04$ sec (blue).}
\end{figure}
\looseness-1From the results presented in Figures\ \ref{fig:angSpeed_dt}, \ref{fig:angSpeed_rho}, \ref{fig:kkt} and \ref{fig:feas}, one may conclude that sampling faster does not necessarily result in better performance of Algorithm \ref{algo:splt_trck}.~This behaviour is confirmed by Figure\ \ref{fig:par_dc_palm}.~For every computational power within $\left\{1\cdot10^3,2\cdot10^3,3\cdot10^3,4\cdot10^3\right\}$, the sampling period is made vary from ${\Delta}t=2\cdot10^{-3}$ sec to ${\Delta}t=4\cdot10^{-2}$ sec.~The tracking performance is assessed by computed the normalised $L^{2}$-norm of the difference between the full-NMPC output trajectory obtained with\ \textsc{ipopt}\ \cite{waech2006} and the output signal obtained with Algorithm\ \ref{algo:splt_trck} (at the same sampling period), on a fixed time interval between $2$ sec and $4$ sec.~More precisely, the optimality tracking error is defined by  
\begin{align}
E:=\sqrt{\frac{1}{N_s}\sum_{k=1}^{N_s}\left(y^\ast_k-\bar{y}_k\right)^2}\enspace,\nonumber
\end{align}
where $\left\{y^\ast_k\right\}$ is the system output signal obtained with\ \textsc{ipopt}, $\left\{\bar{y}_k\right\}$ is the system output signal obtained with Algorithm\ \ref{algo:splt_trck} (for the same ${\Delta}t$) and $N_s$ is the number of time samples.~For a fast sampling, the error $E$ appears to be quite large ($1\cdot10^0$), as the warm-starting point is close to the optimal solution but only few primal proxes can be evaluated, resulting in little improvement of the initial guess in terms of optimality.~This effect can even be justified further by Theorem\ \ref{th:wk_con}: as the number of primal iterations $M$ is fixed by the sampling period, the term $M^{-\psi\left(d_L,n_z\right)}$ in the expression of $\beta_w\left(\rho,M\right)$ and $\beta_s\left(\rho,M\right)$ is not small enough to dampen the effect of the term $1+\rho\lambda_G$, and thus the contraction\ \eqref{eq:wk_con} becomes looser, thus degrading the closed-loop performance.~As the sampling becomes slower, more primal proximal iterations can be carried out and subsequently, the error $E$ is reduced.~The same reasoning as before on $\beta_w\left(\rho,M\right)$ and $\beta_s\left(\rho,M\right)$ can be made.~However, if the sampling frequency $\nicefrac{1}{{\Delta}t}$ is too low, the initial guess is very far from the optimal point, to the point that Assumption\ \ref{ass:prm_diff} may not be satisfied anymore, hence the error increases again.~Thus, at every computational power, an optimal sampling period is obtained.~As the computation power increases, the optimal ${\Delta}t$ appears to decrease and the associated optimality tracking error $E$ drops. 
\begin{rk}
\looseness-1Note that we compare the behaviour of our parametric optimisation algorithm on NLPs of fixed dimension, no matter what the sampling period is, as the number of prediction samples has been fixed.~This means that the prediction time changes as the sampling period varies, which may have an effect on the closed-loop behaviour.~However, it is important to remember that the error $E$ is measured with respect to the closed-loop trajectory under the optimal full-NMPC control law computed at the same sampling period.
\end{rk}   
\begin{figure}[h!]
\begin{center}
\includegraphics[width=1\columnwidth]{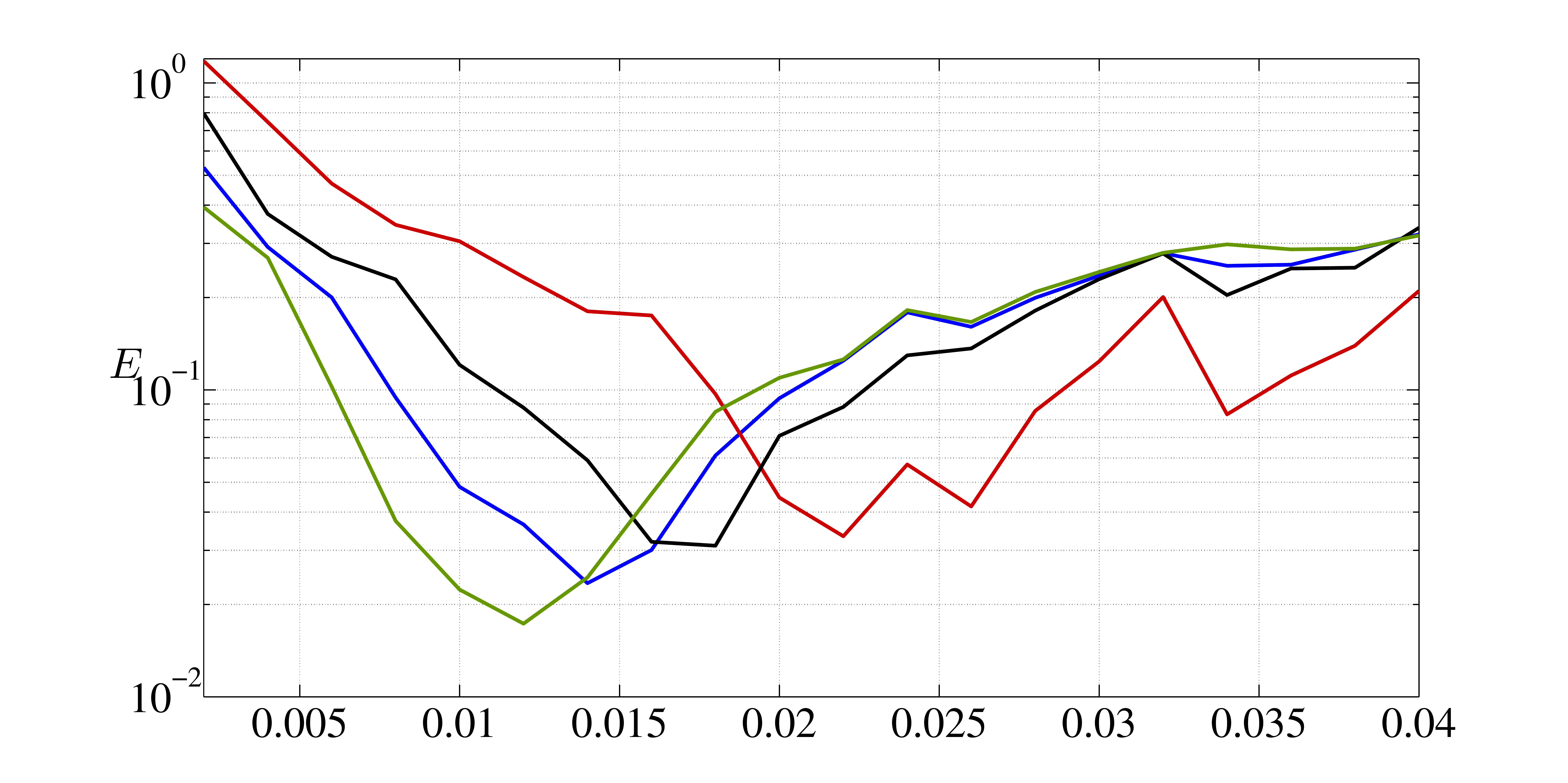}
\vspace{-0.4cm}
\put(-25,5){\footnotesize Sampling period ${\Delta}t$ (s)}
\end{center}
\caption{\label{fig:par_dc_palm}Evolution of the optimality tracking error $E$ against sampling period for different computation power: $1{\cdot}10^3$ primal iterations per sec.(red), $2{\cdot}10^3$ (black), $3{\cdot}10^3$ (blue) and $4{\cdot}10^3$ (green).}
\end{figure}
\begin{figure}[h!]
\begin{center}
\includegraphics[width=1\columnwidth]{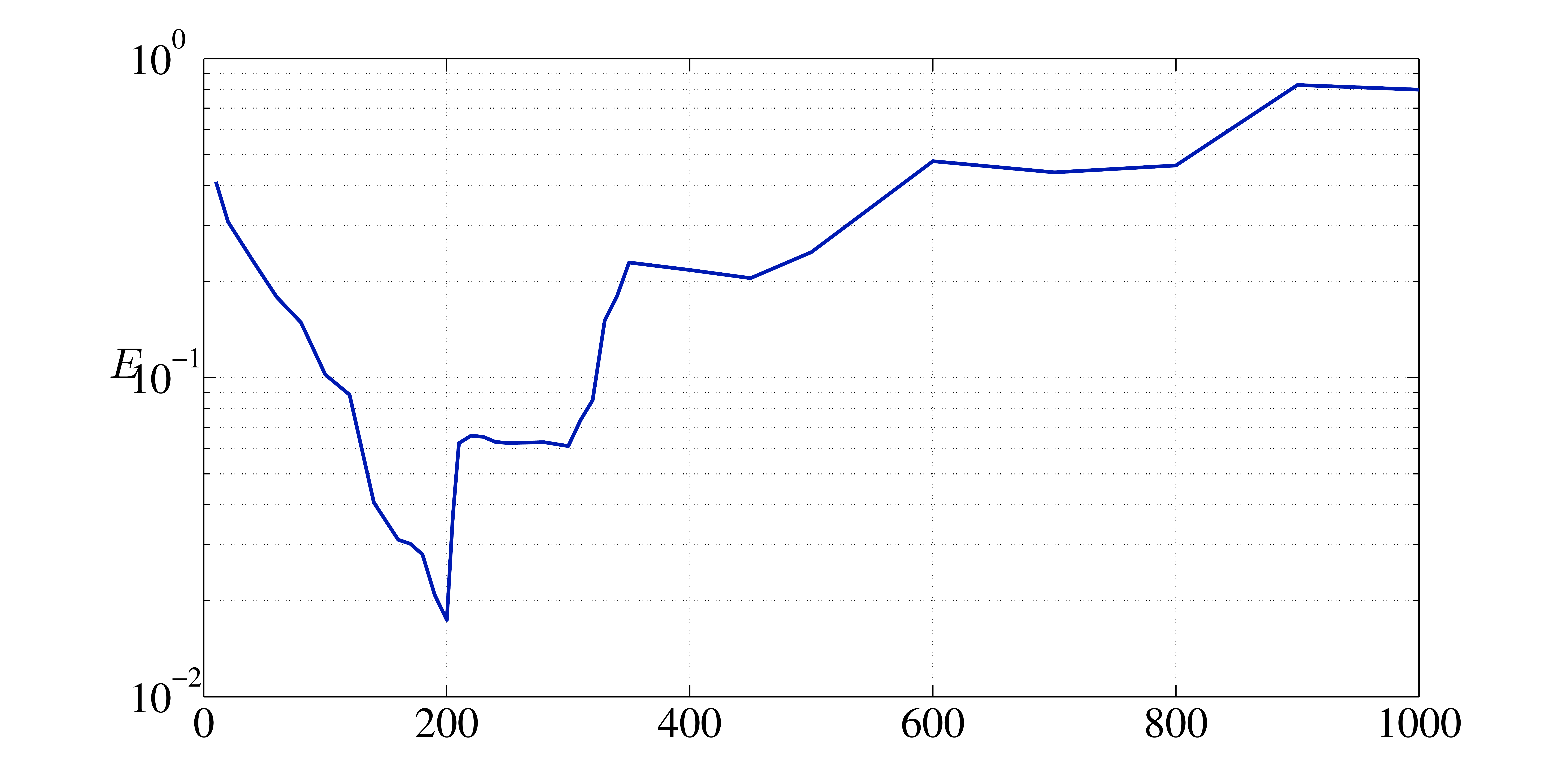}
\vspace{-0.4cm}
\put(-15,5){\footnotesize Penalty parameter $\rho$}
\end{center}
\caption{\label{fig:rho_eff}Evolution of the optimality tracking error $E$ against penalty parameter $\rho$ for $2\cdot10^3~\nicefrac{\text{prox}}{\text{sec}}$ and ${\Delta}t=0.018$ sec.}
\end{figure}
\looseness-1An interesting aspect of the non-convex splitting Algorithm\ \ref{algo:splt_trck} is that the step-size $\rho$ has an effect on the closed-loop behaviour of the nonlinear dynamics, as shown in Fig.\ \ref{fig:rho_eff}.~Given fixed sampling period and computational power, the tracking performance can be improved by tuning the optimiser step-size $\rho$.~In a sense, $\rho$ can now be interpreted as a tuning parameter for the NMPC controller.~In particular, for a fixed number of primal iterations $M$, choosing $\rho$ too large makes the numerical value of the contraction coefficients $\beta_w\left(\rho,M\right)$ and $\beta_s\left(\rho,M\right)$ blow up, subsequently degrading the tracking performance. 
\begin{figure}[h!]
\begin{center}
\includegraphics[width=1\columnwidth]{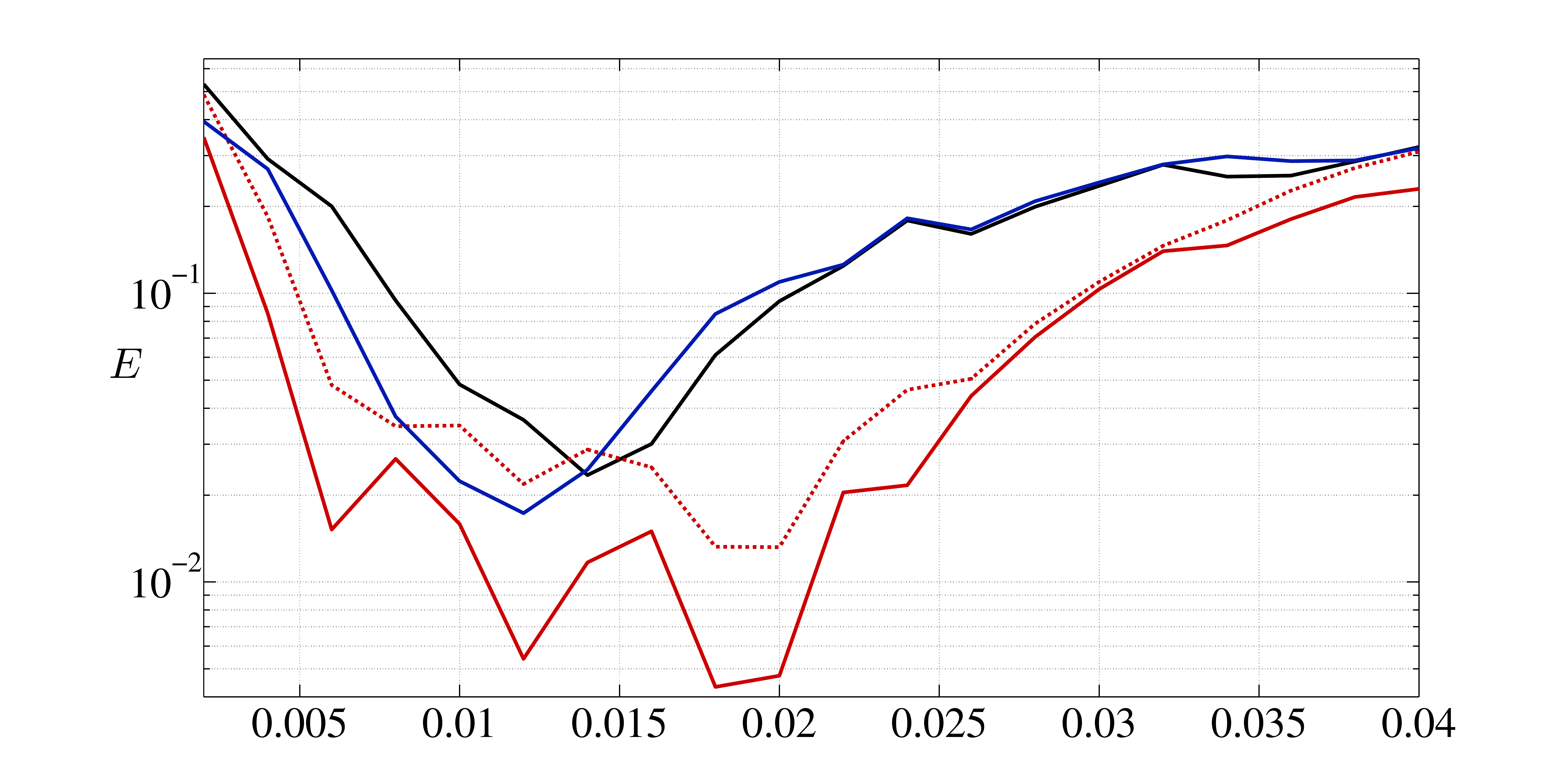}
\vspace{-0.4cm}
\put(-15,5){\footnotesize Sampling period ${\Delta}t$ (s)}
\end{center}
\caption{\label{fig:par_hom_palm}Evolution of the optimality tracking error $E$ against sampling period ${\Delta}t$.~Algorithm\ \ref{algo:splt_trck} for $3\cdot10^3~\nicefrac{\text{prox}}{\text{sec}}$ in black, for $4\cdot10^3~\nicefrac{\text{prox}}{\text{sec}}$ in blue.~Algorithm\ \ref{algo:hom_splt_trck} with $3$ homotopy steps for $3\cdot10^3~\nicefrac{\text{prox}}{\text{sec}}$ in dashed red, with $4$ homotopy steps for $4\cdot10^3~\nicefrac{\text{prox}}{\text{sec}}$ in red.}
\end{figure}

\looseness-1From the arguments developed in paragraph\ \ref{subsec:hom_con} of Section\ \ref{sec:con_prop}, one can expect Algorithm\ \ref{algo:hom_splt_trck} to track the time-dependent optima more closely than Algorithm\ \ref{algo:splt_trck}.~This is confirmed by Fig.\ \ref{fig:par_hom_palm}.  

\subsection{Collaborative tracking of unicycles}
\label{subsec:coll_track}
\looseness-1The second example is a collaborative tracking problem based on NMPC.~Three unicycles are controlled so that a leader follows a predefined path, while two followers maintain a fixed formation.~This control objective can be translated into the cost function of an NMPC problem, which is then written  
\begin{align}
\int_0^T\left\|x^{(1)}\left(t\right)-x^r\left(t\right)\right\|_{Q_1}^2&+\left\|u^{(1)}\left(t\right)\right\|_{R_1}^2+\left\|u^{(2)}\left(t\right)\right\|_{R_2}^2\nonumber\\
		 &+\left\|u^{(3)}\left(t\right)\right\|_{R_3}^2\nonumber\\
		 &+\left\|x^{(1)}\left(t\right)-x^{(2)}\left(t\right)-d_{1,2}\right\|_{Q_{1,2}}^2\nonumber\\
		 &+\left\|x^{(1)}\left(t\right)-x^{(3)}\left(t\right)-d_{1,3}\right\|_{Q_{1,3}}^2\mathrm{d}t,\nonumber
\end{align}
where $Q_1,Q_{1,2},Q_{1,3},R_1,R_2,R_3$ are positive definite matrices, $d_{1,2},d_{1,3}$ are vectors that define the formation between unicycles $1$, $2$ and $3$ and $x^r\left(\cdot\right)$ is a reference path.~All agents $1$, $2$ and $3$ follow the standard unicycle dynamics
\begin{align}
\left\{
\begin{aligned}
\dot{x}_1&=u_1\cos{x_3}\\
\dot{x}_2&=u_1\sin{x_3}\\
\dot{x}_3&=u_2
\end{aligned}
\right.\enspace,\nonumber
\end{align}
subject to input constraints 
\begin{align}
\nonumber
u_1\in\left[0,0.5\right]\enspace,~~~u_2\in\left[-\displaystyle\frac{\pi}{2},\displaystyle\frac{\pi}{2}\right]\enspace.
\end{align}
\begin{figure}[h!]
\begin{center}
\includegraphics[scale=0.1]{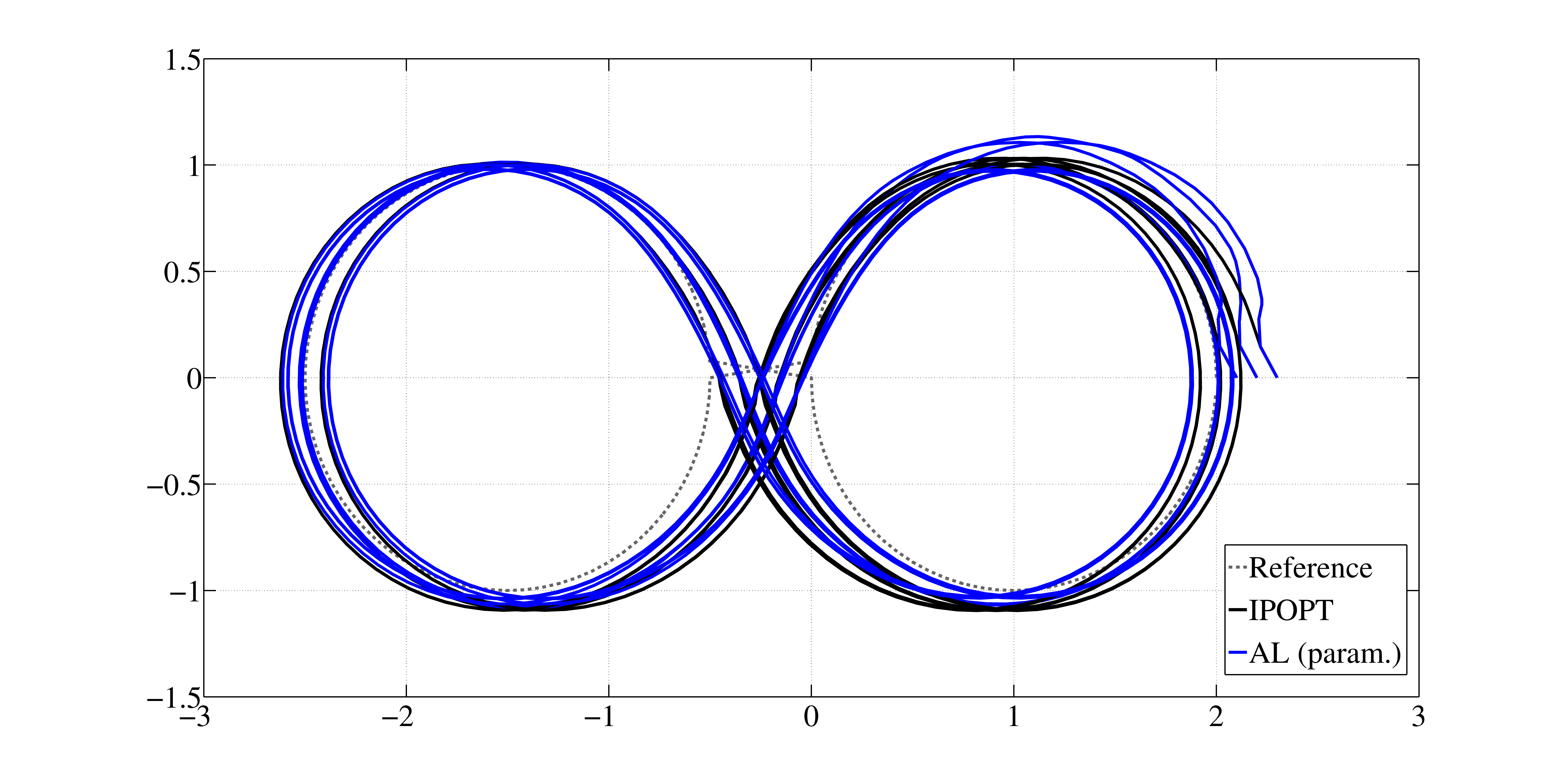}
\end{center}
\caption{\label{fig:traj_coll_track}Trajectories of the three-unicycles formation for $300~\nicefrac{\text{prox}}{\text{sec}}$, ${\Delta}t=0.35$ sec and $\rho=2\cdot10^3$.}
\end{figure}
\looseness-1The continuous-time NMPC problem is discretised using a Runge-Kutta integrator of order $4$\ \cite{hai2008}, while the closed-loop system is simulated with the\ \textsc{matlab} varying step-size integrator \texttt{ode45}.~In the resulting finite-dimensional NLP, two cost coupling terms appear between agents $1$ and $2$, and agents $1$ an $3$.~This can be addressed by the splitting Algorithm\ \ref{algo:splt_trck}.~Moreover, the whole procedure then consists in a sequence of proximal alternating steps between agent $1$  and the group $\left\{2,3\right\}$, which can compute their proximal descents in parallel without requiring any communication.~For this particular NLP with cost-couplings, the dual updates can be performed in parallel.
\begin{figure}[h!]
\begin{center}
\includegraphics[width=1\columnwidth]{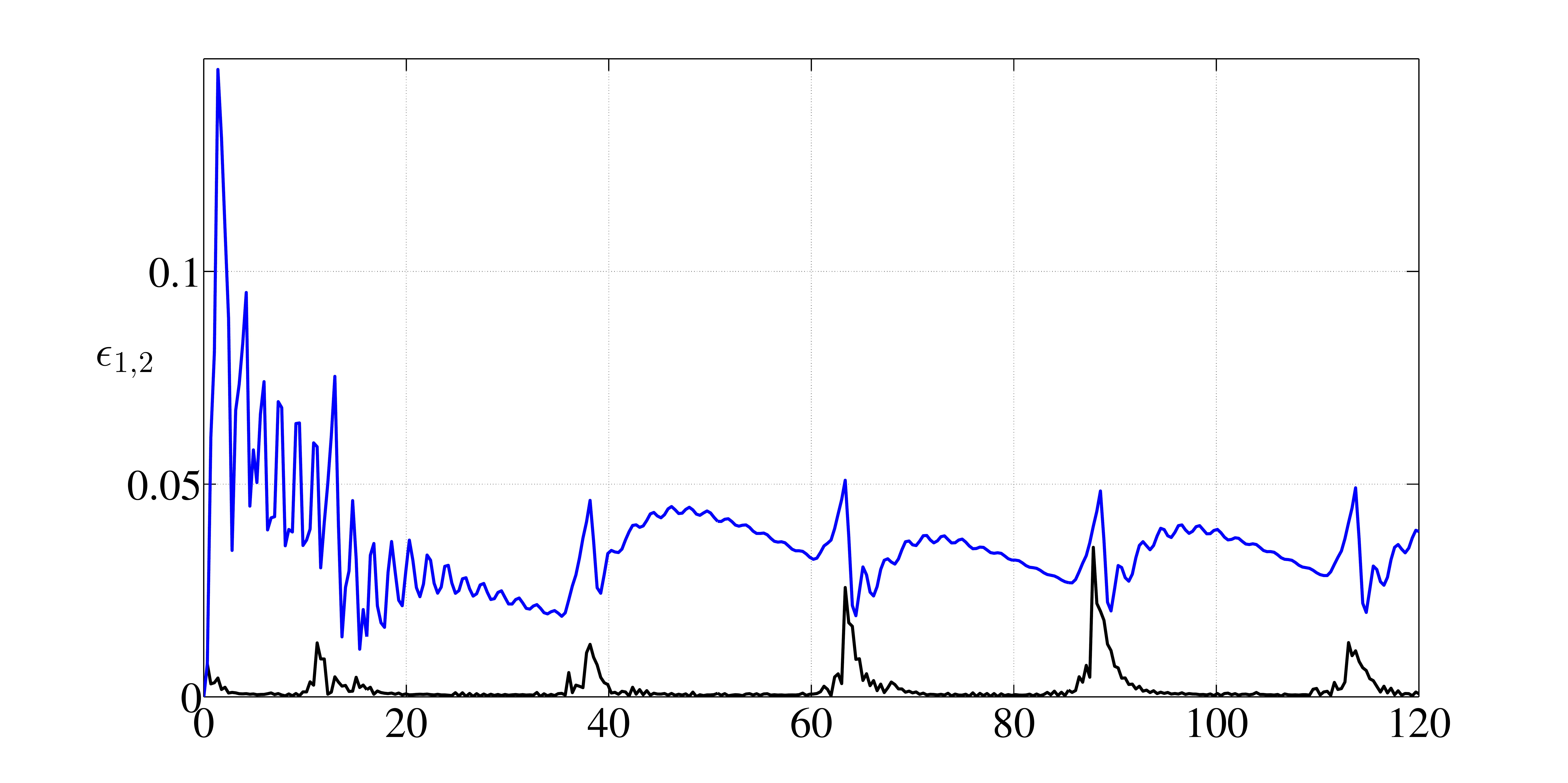}
\vspace{-0.4cm}
\put(-10,5){\footnotesize Time (s)}
\end{center}
\caption{\label{fig:form_err}Evolution of the formation error between unicycles $1$ and $2$ for Algorithm\ \ref{algo:splt_trck} (blue), compared with the formation error obtained with the full NMPC (\textsc{ipopt}, black).}
\end{figure}
\looseness-1Results of the collaborative tracking NMPC are presented in Figures\ \ref{fig:traj_coll_track} and \ref{fig:form_err}.~The number of iterations/communications per second has been fixed at $300$ and the sampling period set to ${\Delta}t=0.35$ sec.~Within the sampling period, this results in $M=105$ exchanges of packets between agent $1$ and agents $2$,$3$, which perform their computations in parallel.~The penalty parameter was $\rho=2\cdot10^3$.~The formation-keeping NMPC has been first simulated with the unicycles in closed-loop with the full-NMPC control law, computed using\ \textsc{ipopt} with accuracy $1\cdot10^{-7}$, which is purely centralised, hence not very interesting from a practical point of view, in this particular case.~The full-NMPC trajectory is plotted in black in Fig.\ \ref{fig:traj_coll_track}, while the one obtained using Algorithm \ref{algo:splt_trck} is represented in blue.~The closed-loop formation error 
\begin{align}
\epsilon_{1,2}:=\left\|x^{(1)}-x^{(2)}-d_{1,2}\right\|_2\nonumber
\end{align}
is plotted in Fig.\ \ref{fig:form_err}. At every reference change, the error rises, but decreases again as the tracking converges.~The performance could be further improved by tuning the penalty $\rho$ or performing a few homotopy steps as in Algorithm\ \ref{algo:hom_splt_trck}. 
\section{Conclusion}
\label{sec:conc}
An novel non-convex splitting Algorithm for solving NMPC programs in a real-time framework has been presented.~Contraction of the primal-dual sequence has been proven using regularity of generalised equations and recent results on non-convex descent methods.~It has been shown that the proposed Algorithm can be further improved by applying a continuation technique.~Finally, the proposed approach has been analysed in details on two numerical examples.  

\bibliographystyle{plain}	
\bibliography{biblio}	
    
\end{document}